\documentclass{amsart}[12pt]

\usepackage[letterpaper,top=2cm,bottom=2cm,left=3cm,right=3cm,marginparwidth=1.75cm]{geometry}

\usepackage{amsmath, eucal}
\usepackage{accents}
\usepackage{graphicx, tikz}
\usepackage[colorlinks=true, allcolors=blue]{hyperref}
\usepackage{enumitem}

\numberwithin{equation}{section}

\newtheorem{theorem}{Theorem}[section]

\newtheorem{corollary}[theorem]{Corollary}
\newtheorem{conjecture}[theorem]{Conjecture}
\newtheorem{proposition}[theorem]{Proposition}

\newtheorem{lemma}[theorem]{{\bf Lemma}}

\newcounter{hyp_counter}
\theoremstyle{definition}

\newtheorem{definition}[theorem]{Definition}%

\theoremstyle{remark}
\newtheorem{remark}[theorem]{Remark}

\newtheorem{theoalph}{\bf Theorem}

\newcommand{\R}{\mathbb R}
\newcommand{\T}{\mathbb T}
\newcommand{\Z}{\mathbb Z}
\newcommand{\eps}{\varepsilon}
\newcommand{\cU}{\mathcal U}
\newcommand{\cV}{\mathcal V}
\newcommand{\cW}{\mathcal W}
\newcommand{\cT}{\EuScript T}

\renewcommand{\phi}{\varphi}

\title{Deformation and perturbative rigidity near de la Llave examples}

\author{Andrey Gogolev}
\thanks{The first author was partially supported by the NSF grant DMS-2247747.}
\address{Department of Mathematics, The Ohio State University\\ Columbus, OH 43210, USA.}
\email{gogolyev.1@osu.edu}

\author{Martin Leguil}
\thanks{The second author was partially supported by the ANR AAPG 2021 PRC CoSyDy (Grant No. ANR-CE40-0014), by the ANR JCJC PADAWAN (Grant No. ANR-21-CE40-0012), by the ANR NO-LIMIT, and by the LESET Math-AMSUD project.}
\address{\'Ecole polytechnique, CMLS\\
	Route de Saclay, 91128 Palaiseau Cedex, France.}
\email{martin.leguil@polytechnique.edu}

\begin{document}

\begin{abstract}
De la Llave’s examples are Anosov diffeomorphisms on the four-torus $\mathbb{T}^4$ with constant Lyapunov spectrum, yet they are not $C^1$-conjugate to the linear model or to each other. Nevertheless, we show that such examples are ``locally exceptional'': we prove deformation and local rigidity for generic diffeomorphisms in proximity of de la Llave’s examples.
\end{abstract}

\maketitle

\section{Introduction}

This paper is devoted to deformation and perturbative rigidity of Anosov diffeomorphisms in dimension four. Consider an Anosov diffeomorphism $F_0$ on a compact Riemannian manifold $M$. In other words, there exists a $DF_0$-invariant splitting $TM=E^s \oplus E^u$ of the tangent bundle into stable and unstable spaces $E^s$, $E^u$, such that vectors in $E^s$ (resp. $E^u$) get exponentially contracted under forward (resp. backward) iteration of $DF_0$. By structural stability, for any diffeomorphism $G$ which is sufficiently $C^1$-close to $F_0$, there exists a homeomorphism $h$, called a (topological)~\emph{conjugacy}, such that $h\circ F_0=G\circ h$. Yet, there exist plenty of obstructions to the existence of a $C^1$ conjugacy; indeed, if the conjugacy map $h$ above can be chosen to be $C^1$, then for any periodic point $p=F_0^n(p)$, we can differentiate the conjugacy equation and obtain
$$
Dh(p)DF_0^n(p)(Dh(p))^{-1}=DG^n(h(p)),
$$
that is, the differentials $DF_0^n(p)$ and $DG^n(h(p))$ must be conjugate. In fact, for any periodic point $p$, it is easy to arrange the perturbation $G$ such that the above condition fails (and this is typical). Therefore, there are countably many obstructions to lift to hope for the existence of a $C^1$ conjugacy, which are associated to the periodic points. 

We say that $F_0$ and $G$ are~\emph{isospectral} if for every periodic point $p=F_0^n(p)$ the linearized return maps $DF_0^n(p)$ and $DG^n(h(p))$ have the same collection of eigenvalues. In particular, given a one-parameter family\footnote{We assume continuity in $C^2$ topology.} of Anosov diffeomorphisms $\{F_s\}_{s\in[0,1]}$, there exists a unique continuous family of  conjugacies $\{h_s\}_{s\in[0,1]}$ such that $h_s\circ F_0=F_s\circ h_s$ and $h_0=\mathrm{id}$. Accordingly, such family is called~\emph{isospectral} if for every periodic point $p=F_0^n(p)$, all the linearized return maps $DF_s^n(h_s(p))$, $s \in [0,1]$, have the same collection of eigenvalues.  

It is a classical result that for $2$-dimensional Anosov diffeomorphisms, the collection of eigenvalues of periodic points is actually a complete set of moduli of smooth conjugacy classes:
\begin{theorem}[Marco-Moriy\'on, de la Llave~\cite{InvI,InvII,InvIII,dlLSRB}]\label{marco_moriyon_delallave}
	Let $F_0$, $G$ be two $C^r$, $r \in(1,\infty]\cup\{\omega\}$, Anosov diffeomorphisms on $\mathbb{T}^2$ which are topologically conjugated, and whose eigenvalues at corresponding periodic points match. Then the conjugacy is $C^{r_*}$ regular, with \begin{equation*} 
r_*=r,\text{ if }r\notin \mathbb{N},\text{ and }r_*=(r-1)+\mathrm{Lip}, \text{ if }r\in \mathbb{N}.
\end{equation*} 
\end{theorem}

In particular, in this low-dimensional situation, once the conjugacy is $C^1$, due to the bootstrap phenomenon, it is actually as regular as the Anosov diffeomorphisms themselves.
In higher dimension, it is not always possible to bootstrap the regularity of a $C^1$ conjugacy between Anosov diffeomorphisms. The lack of bootstrap phenomenon is typically related to the lack of regularity of the leaves of intermediate (weak stable/unstable) foliations, when they exist.  

\begin{definition}  A $C^r$ Anosov diffeomorphism $F_0$ is called $C^r$-\emph{locally rigid}, $r \geq 2$, if there exists a neighborhood $\mathcal{U}$ of $F_0$ in the $C^r$ topology such that for any diffeomorphism $G\in \mathcal{U}$ that is isospectral to $F_0$, the corresponding conjugacy $h$ is, in fact, a $C^{1+\textup{H}}$ diffeomorphism.\footnote{Here, $C^{1+\textup{H}}$ means $C^1$ with H\"older continuous differential.} Similarly, $F_0$ is called $C^r$-\emph{deformation rigid} if for any isospectral one-parameter family $\{F_s\}_{s\in[0,1]}$ based at $F_0$ such that the map $ s\mapsto F_s$ is continuous in the $C^r$ topology, the corresponding family of conjugacies $h_s$ is, in fact, a family of $C^{1+\textup{H}}$ diffeomorphisms.
\end{definition}

In higher dimension, there exist counterexamples to the deformation and local rigidity of Anosov diffeomorphisms constructed by de la Llave~\cite{dlLSRB}, which we now proceed to recall. 

\subsection{de la Llave examples}\label{subs_DeLaLlave}
 Let $A$ and $B$ be automorphisms of the 2-torus $\T^2$ induced by hyperbolic matrices in $\mathrm{SL}(2,\Z)$. We will assume that the smaller eigenvalues $\lambda$ and $\mu$ of $A$ and $B$, respectively, satisfy the following inequalities: $0<\lambda<\mu<1$. Define $\alpha=\log\mu/\log\lambda$ and notice that $\alpha\in(0,1)$. 

 \begin{definition}
  A {\it de la Llave diffeomorphism} $L_{\phi_0}\colon \T^4\to\T^4$ is defined as a skew-product
$$
L_{\phi_0}(x,y)=(Ax, By+\phi_0(x)), \quad\forall\, (x,y)\in\T^2\times\T^2,
$$
where $\phi_0\colon \T^2\to\T^2$ is a smooth function.
\end{definition}
Such diffeomorphism is Anosov and, if $\phi_0$ is homotopic to a constant, is conjugate to the linear product automorphism $L_0$. More generally, if $\phi_1$ is homotopic to $\phi_0$ then the corresponding de la Llave diffeomorphism $L_{\phi_1}$ is conjugate to $L_{\phi_0}$ with conjugacy $h$ given by
\begin{equation}\label{conj_map}
h(x,y)=(x,y+\psi(x)), \quad\forall\, (x,y)\in\T^2\times\T^2,
\end{equation}
where $\psi\colon\T^2\to\T^2$ is null-homotopic. 
Indeed, a solution of the form in~\eqref{conj_map} to the conjugacy equation $h\circ L_{\varphi_0}=L_{\varphi_1}\circ h$ can be obtained explicitly by solving the cohomological equation for $\phi_1-\phi_0$:
\begin{equation}\label{co_eq_conj_map}
(\varphi_1-\varphi_0)(x)=\psi\circ A(x) - B \circ \psi(x),\quad \forall\, x\in \mathbb{T}^2.
\end{equation}
We can check that this equation has a unique solution $\psi\in C^\alpha(\T^2)$, and then, the associated map $h$ is the unique conjugacy between $L_{\varphi_0}$ and $L_{\varphi_1} $ in the homotopy class of the identity, and it is also $C^\alpha$ regular. However, for some simple choices of $\phi_1-\phi_0$, and even for a generic choice, the function $\psi$ is not  $C^{\alpha+\eps}$ regular, for any $\eps>0$. Let us, for example, consider the case where $\varphi_1-\varphi_0$ is a single Fourier mode $x\mapsto e_u \cos (2\pi \langle k_0,x\rangle)$, $k_0 \in \mathbb{Z}^2\setminus \{(0,0)\}$, where $\langle\cdot,\cdot \rangle$ is the function on $\mathbb{T}^2 \times \mathbb{T}^2$ induced by the Euclidean inner product on $\mathbb{R}^2$, and $e_u$ is an unstable eigenvector for $B$, $B e_u=\mu^{-1}e_u$. The solution $\psi$ to~\eqref{co_eq_conj_map} is of the form $\tilde{\psi}e_u$, for some function $\tilde{\psi}\colon \mathbb{T}^2 \to \mathbb{R}$. Considering the Fourier series $\sum_{k\in \mathbb{Z}^2} \hat \psi_k  e^{\mathrm{i} 2\pi \langle k,x\rangle}$ of $\tilde{\psi}$, 
we find that 
$$
\hat \psi_{\pm (A^\top)^{n}k_0}=-\mu^{n+1},\quad \forall\, n\geq 0,\qquad \hat \psi_{k}=0,\quad \forall\, k \notin \{\pm (A^\top)^{n} k_0\}_{n\geq 1}\cup \{(0,0)\}.
$$
For some constant $K>0$, we have $|(A^\top)^{n}k_0|\geq K \lambda^{-n}$, hence the decay of the Fourier coefficients ensures that $\tilde{\psi}\in C^\alpha$ but  $\tilde{\psi}\notin C^{\alpha+\eps}$, for any $\eps>0$. Consequently, the conjugacy $h$ is merely $C^\alpha$ regular despite the fact that all eigenvalues at corresponding periodic points (and Lyapunov exponents of all corresponding measures) of $L_{\phi_0}$ and $L_{\phi_1}$ are the same. This very interesting phenomenon was discovered and explained by de la Llave in~\cite{dlLSRB}, see also~\cite{Go, GRH} for later expositions. De la Llave diffeomorphisms also demonstrate several other interesting dynamical features. For example, they are non-linear Anosov diffeomorphisms for which the invariant volume measure has maximal entropy; they could be used to exhibit insufficient regularity of weak foliations (when $\alpha>1$), etc.

Furthermore, by linearly extrapolating into a family of diffeomorphisms $\{L_{\phi_s}\}_{s \in [0,1]}$, $\phi_s=\phi_0+s(\phi_1-\phi_0)$, $s\in [0,1]$, we obtain an isospectral family, yet the conjugacy between any two diffeomorphisms in the family is merely $C^{\alpha}$.

The first author studied local $C^{1+\textup{H}}$-conjugacy classes of de la Llave diffeomorphisms in~\cite{Go}. In particular, it was established that a $C^{1+\textup{H}}$ isospectral perturbation of a de la Llave diffeomorphism $L_{\phi_0}$ is $C^{1+\textup{H}}$-conjugate to a de la Llave diffeomorphism $L_{\phi_1}$, where $\phi_1$ is $C^1$-close to $\phi_0$. Hence, de la Llave examples completely absorb the failure of periodic eigenvalue rigidity for the product automorphism $L_0$. This naturally raises the question of description of local $C^{1+\textup{H}}$-conjugacy classes for perturbations of $L_0$ and, more generally, of $L_{\phi_0}$.

\subsection{Deformation and perturbative rigidity}

The following theorems constitute progress on a smooth classification in proximity of de la Llave diffeomorphisms. Recently, Rafael de la Llave communicated to the authors that he and others considered the possibility of such rigidity in proximity of the examples at the time when examples were discovered. 

\begin{theoalph}\label{main_theorem}
    Let $L_\phi$ be a de la Llave diffeomorphism. Then there exist a $C^1$ small neighborhood $\cU$ of $L_\phi$ and a $C^2$-open $C^\infty$-dense subset $\cV\subset\cU$ such that each Anosov diffeomorphism $F\in\cV$ is $C^2$-deformation rigid.
\end{theoalph}

\begin{theoalph}\label{main_theorem_bis}
    Let $L_\phi$ be a de la Llave diffeomorphism. Then there exist a $C^1$ small neighborhood $\cU$ of $L_\phi$ and a $C^2$-open $C^\infty$-dense subset $\cV\subset\cU$ such that each Anosov diffeomorphism $F\in\cV$ is $C^2$-locally rigid.
\end{theoalph}

\begin{remark}
   We note that there is no obvious way to deduce Theorem~\ref{main_theorem} from Theorem~\ref{main_theorem_bis}. The issue is that an isospectral deformation is not required to be local and stay in a neighborhood of $F$ (moreover, can leave the neighborhood of $L_\phi$) and, hence, one cannot apply Theorem~\ref{main_theorem_bis} to the path, but only to the beginning of the path. One then could attempt covering the path by a finite number of small neighborhoods and gradually conjugate the whole path to $F$. However, this approach is problematic since the conjugacy we obtain is only $C^{1+\textup{H}}$ regular, and it seems hard to make our proof work in this low regularity predicament. 
\end{remark}

\begin{remark}
To keep the exposition clean, we work with $C^\infty$ diffeomorphisms. But by carefully inspecting the proof, one can check that the same results are true for $C^4$ diffeomorphisms.  
\end{remark}

The latter result suggests the following open questions.
\begin{enumerate}[label=\textbf{(Q\arabic*)},ref=(Q\arabic*)]
\item\label{Qu1} 
{\it Global rigidity.} Show that for $F\in \cV$, any Anosov diffeomorphism $G$ which is isospectral to $F$ is $C^1$-conjugate to $F.$
\item\label{Qu2} 
{\it Volume-preserving 
version.} Show that there exists a $C^1$ small neighborhood $\cU^{\mathrm{vol}}$ 
of $L_\phi$ in the space of volume-preserving 
Anosov diffeomorphisms and a $C^2$-open $C^\infty$-dense subset $\cV^{\mathrm{vol}}\subset\cU^{\mathrm{vol}}$ 
such that each $F\in\cV^{\mathrm{vol}}$ 
is ($C^2$-locally) rigid. 
\item\label{Qu3}  
{\it Jacobian rigidity.} Show that one can weaken the isospectrality assumption in Theorems~\ref{main_theorem} and~\ref{main_theorem_bis} to mere matching of Jacobians at corresponding periodic points.
\end{enumerate}

\begin{remark}
    Let us make some observations on these questions: 
    \begin{itemize}
        \item our strategy relies on certain period expansions near periodic points with dissipative behavior along the center (see e.g. Proposition~\ref{prop_trois_un}); in the volume-preserving case, as suggested in Question~\ref{Qu2}, such expansions look different (see~\cite[Proposition 4.17]{GLRH}), and the leading exponential term mixes data coming from the (weak-)stable and unstable directions, which would make harder the identification of bifurcations such as those we research, e.g., in Proposition~\ref{prop_main_tech_match}; also, considering suspensions via logarithmic full Jacobian is not useful in the volume-preserving setting;
        \item regarding Question~\ref{Qu3}, while the expansions we consider here do assume matching of eigenvalues along the different directions, we may be able to derive such information from matching of the full Jacobian, similarly to what was achieved in~\cite[Theorem~A-Corollary~B]{GLRH}. 
    \end{itemize}
\end{remark}

In fact, some years ago the first author gave a conjectural description of all possible ``generalized de la Llave" examples. While our results give a definitive progress on rigidity near de la Llave examples, they still fall short of such an explicit description. We record this conjecture here.

\begin{conjecture}
    Let $L_\phi$ be a de la Llave diffeomorphism and let $\cU$ be a small $C^1$-neighborhood of it where the $4$-way dominated splitting survives. Let $F\in\cU$ and assume that the center foliation $W^c_F$ by $2$-tori, which is the integral foliation of the weak stable and weak unstable distributions of $F$, is not $C^1$ regular as a foliation.\footnote{See also Remark~\ref{remark_regularity} for additional insights about why anomalous regularity of invariant foliations can be an issue in the study of rigidity questions.} Then $F$ is $C^1$-locally rigid.
\end{conjecture}

\begin{remark}\label{rmk_centerfol}
    The condition ``$W^c_F$ is not $C^1$ regular" can be fully understood in terms of periodic data of $F$. Namely, $W^c_F$ is $C^1$ regular if and only if for any periodic point $p$ and any other periodic point $q\in W^c_F(p)$ the strong stable and strong unstable Lyapunov exponents of $p$ and $q$ are equal. This fact can be derived via a similar argument to the characterization of $C^1$ weak unstable foliations on $\T^3$ given in~\cite{G12}.
\end{remark}

Let us briefly summarize the main steps of the proof of Theorem~\ref{main_theorem_bis}. Fix a diffeomorphism $F \in \mathcal{V}$, and a diffeomorphism $G$ which is sufficiently $C^2$-close to $F$ and isospectral to $F$. 
\begin{enumerate}
    \item We consider the suspension flows $X^t$ and $Y^t$ over $F$ and $G$, respectively, with roof functions given by the logarithmic full Jacobian (+const); by the isospectrality condition, the resulting flows are conjugate, i.e., for some homeomorphism $H$, we have $H\circ X^t=Y^t \circ H$, for all $t \in \mathbb{R}$. 
    \item We fix two periodic points $p=X^{T}(p),\tilde{p}=X^{\tilde{T}}(\tilde{p})$, such that $DX^{T}(p),DX^{\tilde{T}}(\tilde{p})$ expands, resp. contracts volume along the ``center''; given homoclinic points $q$ and $\tilde{q}$ to $p$ and $\tilde{p}$, respectively, we consider sequences $(p_n)$, $(\tilde{p}_n)$ of periodic points in the associated horseshoe and study the asymptotics of their periods as $n \to +\infty$ (see Proposition~\ref{prop_trois_un}). 

    \item For suitable choice of $q,\tilde{q}$, we show that matching of the periods of $p_n$ and $H(p_n)$, resp. $\tilde{p}_n$ and $H(\tilde{p}_n)$ for the flows $X^t$, $Y^t$ forces the conjugacy $H$ to send the strong stable manifold of $p$ to the strong stable manifold of $H(p)$, resp. the strong unstable manifold of $\tilde{p}$ to the strong unstable manifold of $H(\tilde{p})$ (see Propositions~\ref{prop_main}-\ref{prop_main_tilda}). 
    More specifically, the asymptotic formulae for the periods allow us to recover the coefficients (one coefficient for each choice of $q,\tilde q$) by the leading exponential terms. This countable collection of coefficients ``have full knowledge" of (completely determine) the position of the strong unstable manifold of $p$ (resp. $\tilde p$) inside the $2$-dimensional unstable manifold. And matching of the coefficients forces matching of strong stable manifolds under $H$. This step constitutes the main novel technique of this paper.
    \item We show that the above condition implies that $H$ preserves strong stable and strong unstable foliations. 
    \item From the preservation of invariant foliations and the isospectrality condition, we conclude that the conjugacy $H$ is $C^{1+\textup{H}}$ regular, and similarly for the conjugacy between the diffeomorphisms $F$ and $G$. 
\end{enumerate} 

In Appendix~\ref{comments_thmc}, we present a rigidity theorem at the level of $5$-dimensional Anosov flows to elucidate key mechanisms and offer additional insight into the proofs of Theorems~\ref{main_theorem} and~\ref{main_theorem_bis}. Indeed, as outlined in the proof sketch for Theorems~\ref{main_theorem}-\ref{main_theorem_bis}, our approach centers on analyzing the conjugacy of $5$-dimensional Anosov flows that are suspensions over diffeomorphisms near de la Llave’s examples. Some properties demonstrated in this specific context actually extend to a broader class of $5$-dimensional Anosov flows; the purpose of Theorem~\ref{claim_improved_thm}  is to formalize this generalization. 
However, the proofs of Theorems~\ref{main_theorem} and~\ref{main_theorem_bis} remain formally independent of Theorem~\ref{claim_improved_thm}.

\begin{remark}\label{remark_regularity}
Complementing Remark~\ref{rmk_centerfol},  the regularity of invariant foliations also appeared as a possible issue for rigidity in the context of $3$-dimensional Anosov flows, see~\cite[Theorem~E]{GLRH}. In fact, it is directly related to one of the ``genericity'' conditions we impose in the definition of the set $\mathcal{V}$ in Theorems~\ref{main_theorem}-\ref{main_theorem_bis}. Indeed, we require the non-vanishing of a quantity related to certain ``templates'' of the suspension flows $X^t$ and $Y^t$, which can be thought of as the temporal coordinate of some invariant distributions (see  Section~\ref{sect_prepa} for the precise definition of these objects, and Proposition~\ref{prop_trois_quatre} for the associated genericity condition). 
     By following the approach developed in~\cite{GLRH}, 
     we claim 
     that for those $F \in \mathcal{V}$ which are non-rigid, the aforementioned templates of the  suspension flow $X^t$ exhibit anomalous regularity, which itself reflects anomalous regularity of the associated distributions, as described in items~\eqref{anomalous_reg1}-\eqref{anomalous_reg2} of Theorem~\ref{claim_improved_thm}. 
     In particular, the genericity condition we need to impose on $F\in \mathcal{V}$ is related to their Jacobian, to avoid falling into cases~\eqref{anomalous_reg1}-\eqref{anomalous_reg2} in Theorem~\ref{claim_improved_thm}. \vspace{0.2cm}
     
{\bfseries Acknowledgements.}
     The first author was supported by the Simons Fellowship during the 2024-25 academic year. The first author
is grateful for excellent working conditions provided by IHES and by the Mathematics Department at Université Paris-Saclay and especially to Sylvain Crovisier for his hospitality. This paper was greatly influenced by the authors' joint work with Federico Rodriguez Hertz~\cite{GLRH} who we thank for many inspiring discussions as well as comments on the first draft. We also thank Jonathan DeWitt for his insightful feedback on the first draft of this paper.
\end{remark}


%

\section{Preparations and structure of the proof}\label{sect_prepa}

Let us fix an automorphism $L_0=(A,B)$, and a de la Llave diffeomorphism $L_\varphi$ as in Subsection~\ref{subs_DeLaLlave}. 
We first explain how the neighborhood $\cU$ is chosen. The automorphism $L_0$ admits constant dominated splitting $T\T^4=E^{ss}_0\oplus E^{ws}_0\oplus E^{wu}_0\oplus E^{uu}_0$ according to the eigendirections corresponding to $\lambda$, $\mu$, $\mu^{-1}$ and $\lambda^{-1}$, respectively. It is a standard exercise on the cone technique to check that a de la Llave diffeomorphism $L_\phi$ also admits a dominated splitting $E^{ss}_\phi\oplus E^{ws}_\phi\oplus E^{wu}_\phi\oplus E^{uu}_\phi$ with the same exponential rates as $L_0$. In fact $E^{ws}_\phi=E^{ws}_0$ and $E^{wu}_\phi=E^{wu}_0$, but the strong subbundles become different. The domination condition is $C^1$-open, so we can choose an open set $\cU\ni L_\phi$ such that each $F\in\cU$ is Anosov and admits a dominated splitting $E^{ss}_F\oplus E^{ws}_F\oplus E^{wu}_F\oplus E^{uu}_F$ which is partially hyperbolic, with $2$-dimensional center $E^{ws}_F\oplus E^{wu}_F$. Let $E_F^s:=E_F^{ss}\oplus E_F^{ws}$ and $E_F^u:=E_F^{uu}\oplus E_F^{wu}$ be the full ($2$-dimensional) stable and unstable bundles. Since the decomposition $E_F^s\oplus E_F^u$ is hyperbolic, $E_F^s$ and $E_F^u$ integrate uniquely to stable and unstable foliations $W_F^s$ and $W_F^u$, respectively. From partially hyperbolic theory (see~\cite{HirschPughShub}) we have that the strong distributions $E^{ss}_F$ and $E^{uu}_F$ uniquely integrate to strong foliations $W^{ss}_F$ and $W^{uu}_F$, respectively. Furthermore, a standard argument on the universal cover ensures that the weak distributions $E^{ws}_F$ and $E^{wu}_F$ integrate uniquely to weak foliations $W^{ws}_F$ and $W^{wu}_F$, respectively~\cite{GG}. Moreover, $E^{ws}_F$ and $E^{wu}_F$ integrate jointly to a ``center'' foliation by tori, which we denote  by $W^c_F$, whose leaves are subfoliated by the leaves of $W^{ws}_F$ and $W^{wu}_F$; this fact can be established by comparison of divergence rates on the universal cover.\footnote{In fact, if $\varphi$ is sufficiently $C^1$-small, we can also appeal to the general Hirsch-Pugh-Shub machinery: since the center distribution $E^{ws}_0\oplus E^{wu}_0$ of $L_0$ is smooth, by plaque-expansiveness,~\cite{HirschPughShub} guarantees that $L_\varphi$ as well as its perturbations are dynamically coherent; in particular,  $E^{ws}_F\oplus E^{wu}_F$ is uniquely integrable, and by intersecting the leaves of the resulting foliation with the leaves of $W_F^s$ and $W_F^u$ respectively, we obtain foliations integrating $E^{ws}_F$ and $E^{wu}_F$. }

We will need the above foliations to persist through an isospectral deformation which could, \textit{a priori}, leave the neighborhood $\cU$. First, by choosing $\cU$ even smaller if needed, we can guarantee that $F\in\cU$ has narrow periodic data (that is, the Lyapunov exponents at periodic points of $F$ are sufficiently close to the Lyapunov exponents of $L_\phi$). If a diffeomorphism $G$ is conjugate to $F$ and has the same periodic data (such as a diffeomorphism from an isospectral deformation of $F$) then $G$ also admits a dominated splitting by the main result of~\cite{DWG}. 
\begin{lemma} \label{lemma1}
    Let $F$ and $G$ be as above. Then the weak stable distribution $E^{ws}_G$ and weak unstable distribution $E^{wu}_G$ are uniquely integrable. 
\end{lemma}
\begin{lemma} \label{lemma2}
    Let $F$ and $G$ be as above and let $h$ be the conjugacy, $h\circ F=G\circ h$. Then $h(W_F^{ws})=W_G^{ws}$ and $h(W_F^{wu})=W_G^{wu}$. 
\end{lemma}

The proofs of these lemmata are quite standard and we only briefly indicate the argument. For more details the reader could consult~\cite{GG}, where similar arguments are carried out in detail. In the following, we abbreviate as $W_\varphi^*:=W_{L_\varphi}^*$, $*\in \{s,ss,ws,wu,uu,u\}$, the invariant foliations of the de la Llave diffeomorphism $L_\varphi$. 

By the main result of~\cite{DWG} the diffeomorphism $G$ also admits an invariant dominated splitting $E^{ss}_G\oplus E^{ws}_G\oplus E^{wu}_G\oplus E^{uu}_G$. Moreover, the exponential contraction and expansion rates along these subbundles are confined to small intervals $(\lambda-\eps,\lambda+\eps)$, $(\mu-\eps,\mu+\eps)$, $(\mu^{-1}-\eps,\mu^{-1}+\eps)$ and $(\lambda^{-1}-\eps,\lambda^{-1}+\eps)$, respectively.

Denote by $h_G$ the conjugacy to $L_\phi$, $h_G\circ L_\phi=G\circ h_G$. If $E^{wu}_G$ does not integrate uniquely it still has one-dimensional integral curves $\gamma^{wu}$ with $\dot\gamma^{wu}\in E^{wu}_G$ inside the leaves $W^u_G$ because $E^{wu}_G$ is a one-dimensional distribution. For any such curve $\gamma^{wu}$ through a point $x$ we have $h_G(\gamma^{wu})\subset W^{wu}_{\phi}(h_G(x))$. This fact follows easily from the following observations:
\begin{enumerate}
\item $\mathrm{length}(G^n(\gamma^{wu}))\le (\mu^{-1}+\eps)^n \mathrm{length}(\gamma^{wu})$;
\item the lift of the conjugacy $h_G$ to $\mathbb R^4$ is a bounded distance away from $\mathrm{id}_{\R^4}$;
\item if $p\in W^{u}_\phi(q)$ and $p\notin W^{wu}_\phi(q)$ then $d_{W^u_\phi}(L_\phi^n(p),L_\phi^n(q))\asymp \lambda^{-n}$, $n\to\infty$.
\end{enumerate}

If $E^{wu}_G$ is not uniquely integrable then it admits branching integral curves $\gamma^{wu}$ through a point $x$ all of which must collapse to the same leaf of $W^{wu}_\phi$: $h_G(\gamma^{wu})\subset W^{wu}_\phi(x)$, contradicting the fact that $h_G$ is a homeomorphism. This proves Lemma~\ref{lemma1} (an analogous argument by iterating backward proves the unique integrability of $E^{ws}_G$). 

To establish Lemma~\ref{lemma2} notice that we proved that $h_G(W^{wu}_G)=W^{wu}_\phi$. The same argument also gives $h_F(W^{wu}_F)=W^{wu}_\phi$, where $h_F$ is the conjugacy between $F$ and $L_\phi$. Hence, $h(W^{wu}_F)=h_G^{-1}\circ h_F(W^{wu}_F)=h_G^{-1}(W^{wu}_\phi)=W^{wu}_G$.

\subsection{Geometric mechanism of de la Llave examples}\label{subs_geometric_mechanism}
Observe that if two Anosov diffeomorphisms are $C^1$ conjugate then the differential of the conjugacy maps the fine dominated splitting of the first diffeomorphism to a fine dominated splitting of the second diffeomorphism. Consequently, the conjugacy maps the ``web'' of invariant foliations of the first diffeomorphisms to that of the second.

Accordingly, if one would like to establish $C^1$ regularity of the conjugacy it is natural to establish matching of the ``webs'' of foliations first. In fact, in the context of de la Llave examples, matching of the strong foliations is the main step to obtain rigidity.

\begin{lemma}\label{lemma3}
    Let $L_\phi$ be a de la Llave diffeomorphism. Then there exists a $C^1$-neighborhood $\cU\ni L_\phi$ such that if $F\in\cU$ and $G$ is conjugate to $F$ via $h$ with the same periodic data and such that $h(W_F^{ss})=W_G^{ss}$, $h(W_F^{uu})=W_G^{uu}$, then $h$ is $C^{1+\textup{H}}$ regular.
\end{lemma}
For a detailed proof we refer to~\cite[Theorem~D]{Go}. The setting in~\cite{Go} is more restrictive (only a neighborhood of the linear automorphism is considered), however, given Lemma~\ref{lemma2} the same proof as in~\cite{Go} gives Lemma~\ref{lemma3}. 
Let us simply recall the main steps in the proof of such result:
\begin{itemize}
    \item for $*\in \{ss,ws,wu,uu\}$, there exist affine structures on the foliation $W_F^*$, and accordingly for the diffeomorphism $G$; more precisely, there exists a continuous family $\{\Psi_{F,x}^*\}_{x \in M}$ of one-dimensional normal forms which linearize the dynamics along the leaves of $W_F^*$, i.e., for each $x \in M$, $\Psi_{F,x}^*\colon W_F^*(x)\to \mathbb{R}$ is a $C^{1+\textup{H}}$ diffeomorphism\footnote{In fact, $\Psi_{F,x}^*$ is smooth for $*\in \{ss,uu\}$; the lack of smoothness of $\Psi_{F,x}^*$ for $*\in \{ws,wu\}$ is due to the lack of regularity of the leaves of $W_F^*$, which are only $C^{1+\textup{H}}$ regular in general.} such that 
    \begin{equation*}
		\Psi_{F,x}^*\circ F=\|DF(x)|_{E_F^*}\|\cdot\Psi_{F,F(x)};
	\end{equation*}
    in fact, for $*\in \{wu,uu\}$, $\Psi_{F,x}^*$ is defined as (the expression for $*\in \{ss,ws\}$ is similar, reversing future and past)
    \begin{equation*}
	\Psi_{F,x}^*(y):=\int_{x}^{y} \rho_F^*(x,z)\, dm_{W_F^*(x)}(z),\quad \forall\, y \in W_F^*(x),
\end{equation*} 
where the above integral is taken over the piece of $W_F^*(x)$ from $x$ to $y$, $dm_{W_F^*(x)}$ being the induced volume on the leaf $W_F^*(x)$, and where
$$
\rho_F^*(x,z):=\prod_{k=1}^{+\infty} \frac{\|DF(F^{-k}(x))|_{E_F^*}\|}{\|DF(F^{-k}(z))|_{E_F^*}\|},\quad \forall\, z \in W_F^*(x);
$$
\item if $F,G$ are isospectral, then by Livshits theorem, the functions $\|DF|_{E_F^*}\|$ and $\|DG\circ h|_{E_F^*}\|$ are cohomologous, i.e., for some Hölder continuous function $\gamma^*\colon M \to \mathbb{R}$, we have 
$$
\|DG(h(x))|_{E_F^*}\|=e^{\gamma^*(F(x))-\gamma^*(x)}\|DF(x)|_{E_F^*}\|,\quad \forall\, x \in M;
$$ 
if, moreover, the conjugacy $h$ sends $W_F^*$ to $W_G^*$ (it is always true for $*\in \{ws,wu\}$), then for any $x \in M$, $z\in W_F^*(x)$, we have  
$$
\rho_G^*(h(x),h(z))=e^{\gamma^*(x)-\gamma^*(z)}\rho_F^*(x,z);
$$
then, using the affine structures on $W_F^*$, it is possible to show that 
the restriction of $h$ to the leaves of $W_F^*$ is uniformly Lipschitz continuous;
\item following the construction of Pesin-Sinai~\cite{PesinSinai},~\cite{GG}-\cite{Go} construct an $F$-invariant measure $\mu$ whose conditionals along the leaves of $W_F^*$ are absolutely continuous with respect to the induced volume, where for $\mu$-a.e. $x \in M$, the density along $W_F^*(x)$ is proportional to $\rho_F^*(x,\cdot)$; although it is not clear that such measure $\mu$ is ergodic, it is shown in~\cite{GG}-\cite{Go} that the set of transitive points has full $\mu$-measure, from which it is possible to upgrade the regularity of $h|_{W_F^*(x)}$ from Lipschitz to
$C^{1+\textup{H}}$;
\item for $*\in \{s,u\}$, the conjugacy $h$ is $C^{1+\textup{H}}$ along the leaves of the two transverse subfoliations of $W_F^{c*}$, hence by Journé's regularity lemma, $h$ is $C^{1+\textup{H}}$ along the leaves of $W_F^*$; since the two foliations $W_F^s$ and $W_F^u$ are transverse, again by Journé's lemma, we conclude that $h$ is $C^{1+\textup{H}}$. 
\end{itemize}

Let $p=F^n(p)$ be a periodic point of $F\in\cU$ and let $\hat\mu_p<\mu_p<\lambda_p<\hat\lambda_p$. We say that $p$ is {\it center-expanding} if $\mu_p\lambda_p>1$ and {\it center-contracting} if $\mu_p\lambda_p<1$. We proceed to state our main technical result which then would supply the hypothesis for Lemma~\ref{lemma3}, which would, in turn, imply the main theorem. 

Given a periodic point $p_0$ we denote by $p_s=F^n_s(p_s)=h_s(p_0)$ its continuation along the family. Also $p_0$ admits a unique continuation to the neighborhood $\cU$ and we will write $p_0^G$ for the corresponding periodic point of $G\in\cU$.

\begin{proposition}\label{prop_main} 
     If $\bar F\in\cU$
    and $p_0$ is a center-expanding periodic point then there exist an arbitrarily $C^\infty$-small perturbation $\hat F$ of $\bar F$ and a $C^2$-open neighborhood $\cW$, $\hat F\in\cW\subset\cU$ such that for all $G\in\cW$ and all isospectral deformations $\{G_s\}_{s \in [0,1]}$ of $G$ we have $h_s(W^{ss}_{G,\mathrm{loc}}(p^G_0))=W^{ss}_{G_s, \mathrm{loc}}(p^G_s)$ for all $s\in[0,1]$.
\end{proposition}

\begin{remark}\label{remark_main}
    By reversing the time we also have the symmetric statement for center-contracting periodic point $q_0$ with the conclusion $h_s(W^{uu}_{G,\mathrm{loc}}(p^G_0))=W^{uu}_{G_s, \mathrm{loc}}(p^G_s)$. 
\end{remark}

We have the following perturbative counterpart of Proposition~\ref{prop_main}:
\begin{proposition}\label{prop_main_tilda} 
     If $\bar F\in\cU$
    and $p_0$ is a center-expanding periodic point then there exist an arbitrarily $C^\infty$-small perturbation $\hat F$ of $\bar F$ and a $C^2$-open neighborhood $\cW$, $\hat F\in\cW\subset\cU$ such that for all $G\in\cW$ and any isospectral diffeomorphism $G_1$ which is sufficiently $C^2$-close to $G$, the conjugacy $h$ between them satisfies  $h(W^{ss}_{G,\mathrm{loc}}(p^G))=W^{ss}_{G_1, \mathrm{loc}}(p^{G_1})$. 
\end{proposition}

\begin{proof}[Proof of Theorem~\ref{main_theorem} assuming Proposition~\ref{prop_main}] The following lemma is the starting point. 
\begin{lemma}\label{lemma4}
  Any $F\in\cU$  admits arbitrarily $C^\infty$-small perturbation $\bar F$ such that $\bar F$ has a center-expanding periodic point $p_0$ and a center-contracting periodic point, whose eigenvalues satisfy non-resonant conditions~\eqref{eq_resonances}. 
\end{lemma}
The proof is very standard via a localized perturbation in the neighborhood of periodic points and we omit it.

We can apply Proposition~\ref{prop_main} to $\bar F$ and $p_0$ given by the lemma to obtain a perturbation $\hat F$ and a $C^2$-open set $\cW(p_0)\subset\cU$, $\hat F\in\cW(p_0)$, such that for all diffeomorphisms $G\in\cW(p_0)$ and all isospectral deformations $G_s$, $s\in[0,1]$, we have $h_s(W^{ss}_{G,\mathrm{loc}}(p^G_0))=W^{ss}_{G_s,\mathrm{loc}}(p^G_s)$. Further, by taking a smaller perturbation if necessary, we can also assume that $\hat F$ still has a center-contracting periodic point which we denote by $q_0$.\footnote{In fact, the construction of perturbation $\hat F$ is localized and the eigenvalues of the center-expanding periodic point won't change under the perturbation.} We can now apply Proposition~\ref{prop_main} again to $\hat F$ and $q_0$ to obtain a perturbation $\tilde F$ and a $C^2$-open set $\cW(q_0)$, $\tilde F\in\cW(q_0)$ such that for all $G\in\cW(q_0)$ all isospectral deformations $G_s$, $s\in[0,1]$ we have $h_s(W^{uu}_{G, \mathrm{loc}}(q^G_0))=W^{uu}_{G_s, \mathrm{loc}}(q^G_s)$, where $q^G_s$ is the continuation of $q^G_0$. Also, note that by choosing the second perturbation to be sufficiently small we can make sure that $\tilde F\in\cW(p_0)$ so that $\cW(F,\bar F,\hat F,\tilde F):=\cW(p_0)\cap \cW(q_0)$ is a non-empty set. Note that for any $G\in\cW(F,\bar F,\hat F,\tilde F)$ and any isospectral deformation $G_s$ we have both $h_s(W^{uu}_{G, \mathrm{loc}}(q^G_0))=W^{uu}_{G_s, \mathrm{loc}}(q^G_s)$ and $h_s(W^{ss}_{G, \mathrm{loc}}(p^G_0))=W^{ss}_{G_s, \mathrm{loc}}(p^G_s)$.

The open set $\cW(F,\bar F,\hat F,\tilde F)$ is near $F$, but $F$ does not belong to its closure. To fix this problem we do the following. For every $k\ge 1$ we can repeat the same construction while ensuring that $d_{C^\infty}(F, \bar F_k)<1/k$, $d_{C^\infty}(\bar F_k, \hat F_k)<1/k$ and $d_{C^\infty}(\hat F_k, \tilde F_k)<1/k$. Then
$$
\cW(F)=\bigcup_{k\ge 1} \cW(F,\bar F_k,\hat F_k,\tilde F_k)
$$
is a $C^2$-open set which contains $F$ in its $C^\infty$-closure. Finally, we let 
$$
\cV=\bigcup_{F\in\cU}\cW(F).
$$
This is the $C^2$-open and $C^\infty$-dense set posited in Theorem~\ref{main_theorem}. Indeed let $G\in\cV$ and let $G_s$ be an isospectral deformation. Since $G\in\cW(F,\bar F_k,\hat F_k,\tilde F_k)$ for some $F$ we have some periodic points $p_0^G$ and $q_0^G$ such that the local strong stable manifolds at $p_s^G$ and local strong unstable manifolds at $q_s^G$ match under the conjugacies. In fact, since we can iterate both backward and forward and the base-points are periodic we immediately deduce matching of global strong manifolds ---   $h_s(W^{ss}_{G}(p^G_0))=W^{ss}_{G_s}(p^G_s)$ and $h_s(W^{uu}_{G}(q^G_0))=W^{uu}_{G_s}(q^G_s)$. Now by~\cite[Section~4.4]{GG} we also have that $h_s(W^{ss}_{G}(x))=W^{ss}_{G_s}(h_s(x))$ for all $x\in W^s(p_0^G)$. Since  the full stable manifold $W^s(p_0^G)$ is dense in $\T^4$ we deduce that $h_s(W^{ss}_G)=W^{ss}_{G_s}$ and, similarly,  $h_s(W^{uu}_G)=W^{uu}_{G_s}$. 

Finally, we can apply Lemma~\ref{lemma3} and conclude that $h_s$ is $C^{1+\textup{H}}$ regular.
\end{proof}

We now have established Theorem~\ref{main_theorem} modulo the proof of Proposition~\ref{prop_main}. By a very similar argument, we can show Theorem~\ref{main_theorem_bis} modulo the proof of Proposition~\ref{prop_main_tilda}. We proceed with some more preparations and then the proof of Propositions~\ref{prop_main}-\ref{prop_main_tilda} in the next section.

\subsection{Passing to suspension flows} Let $F\in\cU$ and let $F_s$, $s\in[0,1]$, be an isospectral deformation of $F_0=F$.   Denote by $J_s\colon\T^4\to\R$ the (full) Jacobian of $F_s$. We pick a sufficiently large constant $K$ such that $K+\log J_s>0$ for all $s\in[0,1]$, and consider the suspension flows $X^t_s\colon M\to M$ of $F_s$ with the roof function $K+\log J_s$. 
The suspensions flow inherit the dominated splitting structure. That is, for all $s\in[0,1]$, and $X^t=X_s^t$ we have an $DX^t$-invariant splitting $TM=E_X^{ss}\oplus E_X^{ws}\oplus\R X \oplus E_X^{wu}\oplus E_X^{uu}$, and we denote the corresponding foliations  by $W^{*}_X$. We also denote by $E_X^{s}:=E_X^{ss}\oplus E_X^{ws}$ (resp. $E_X^{u}:=E_X^{wu}\oplus E_X^{su}$) the full $2$-dimensional stable (resp. unstable) distribution of the flow $X^t$.  
In the following, we will sometimes drop the index $X$ and write $E^*$, $W^*$ instead of $E_X^*$,  $W_X^*$, $*\in \{ss,ws,wu,uu\}$. 

Since the deformation is isospectral, we have that the sum over the period of $\log J_s$ is independent of $s$ for any periodic point of $F_s$. Hence, by the Livshits theorem, we have that flows $X^t_s$ are conjugate to the initial flow $X^t_0$. We denote the conjugacy by $H_s$ --- $H_s\circ X_0^t=X_s^t\circ H_s, t\in\R$. (The conjugacies $H_s$ are unique up to composition with the flow and we can pick a continuous family of conjugacies $H_s$, $s\in[0,1]$.)




\subsection{Local coordinates and templates}\label{subs_localcoords} We consider the deformation $X_s^t$, $s\in[0,1]$, and a periodic point $p_0=X_0^T(p_0)$  with its continuation $p_s$. Note that the period $T$ is shared by all points $p_s$. Denote by $\hat\mu<\mu<1<\lambda<\hat\lambda$ the eigenvalues of $D_{p_s} X_s^T$. Note that all these numbers are independent of $s\in[0,1]$ since the deformation is assumed to be isospectral. From now on, we will focus on various local considerations in the neighborhood of $p_s$. All constructions that will follow as well as the various choices that will be made can be made continuous in $s$, since the hyperbolic structures  depend on the on the dynamics continuously. In order to make the notation lighter, we will omit the index $s$ and simply write $X^t$ for $X_s^t$, $p$ for $p_s$, etc. Only in the last step of the proof will we reintroduce dependence on $s$ into the notation as the presence of a family becomes important.

Let $\Sigma_p$ be a smooth 4-dimensional local transversal to the flow which contains the local stable manifold $W^s_{\mathrm{loc}}(p)$ and the local unstable manifold $W^u_{\mathrm{loc}}(p)$.
Let us denote by $\Lambda=(\lambda_i)_{i=1}^4:=(\hat \mu,\mu,\lambda,\hat \lambda)\in \mathbb{R}^4$ the vector of eigenvalues;  we can assume that there are no resonances of the form 
\begin{equation}
    \label{eq_resonances}
\lambda_i=\Lambda^\alpha,\quad \forall\, i \in \{1,\cdots,4\}, \quad \forall\,\text{multi-indices }\alpha \in \mathbb{Z}^4,\, |\alpha|\leq 3.
\end{equation}
(Indeed, the resonances at a periodic point can be perturbed away and the absence of resonances is an open property.) Then, by Sternberg linearization theorem, there exists a $C^3$ parametrization 
 $\Phi_p\colon (-1,1)^4\to \Sigma_p$ such that the Poincar\'e return map $\Pi_p$ to $\Sigma_p$ is linearized,
\begin{equation}\label{appli_lin}
\mathcal{L}:=\Phi_p^{-1}\circ \Pi_p\circ \Phi_p\colon(\hat\xi,\xi,\eta,\hat\eta)\mapsto(\hat\mu\hat\xi,\mu\xi,\lambda\eta,\hat\lambda\hat\eta).
\end{equation}
The parametrization $\Phi_p$ can be easily extended to an actual chart $\iota_p\colon(-1,1)^5\to M$ around $p=\iota(0,0,0,0,0)$ such that $\Sigma_p=\iota_p((-1,1)^2 \times \{0\}\times (-1,1)^2)$, and for any $(\hat\xi,\xi,t,\eta,\hat\eta)\in (-1,1)^5$, 
\begin{enumerate}
    \item $\iota_p(\hat\xi,\xi,0, \eta,\hat\eta)=\Phi_p(\hat \xi,\xi,\eta,\hat \eta)$;
  \item $\iota_p(\hat\xi,\xi,t,\eta,\hat\eta)=X^t(\iota_p(\hat\xi,\xi,0,\eta,\hat\eta))$.
\end{enumerate}
   

Note that the local weak and strong stable/unstable manifolds through $p$ are just the coordinate axes in this parametrization.

Using this chart we define two functions, the~\emph{weak stable template} $\mathcal{T}_p^{ws}$ and the~\emph{strong stable template} $\mathcal{T}_p^{ss}$, which represent the angular coordinates of (the image in normal coordinates of) the $2$-plane $E_X^s=E_X^{ss}\oplus E_X^{ws}$ as we move along the local unstable manifold $W^u_{\mathrm{loc}}(p)$: 
\begin{equation}\label{template-def}
\begin{array}{r}
    D\iota_p(0,0,0,\eta,\hat \eta)\big(\textup{Span}(0, 1, \cT_p^{ws}(\eta,\hat\eta), *,*)\big)  \subset E_X^{s}(\Phi_p(0,0,\eta,\hat\eta)),\\
    D\iota_p(0,0,0,\eta,\hat \eta)\big(\textup{Span}(1, 0, \cT_p^{ss}(\eta,\hat\eta), *,*)\big)  \subset E_X^{s}(\Phi_p(0,0,\eta,\hat\eta)).
\end{array}
\end{equation} 
\begin{remark}\label{remarque_templates}
    The subspace $D\iota_p(0,0,0,\eta,\hat \eta)\big(\textup{Span}(0, 1, \cT_p^{ws}(\eta,\hat\eta), *,*)\big)$ is close to $E_X^{ws}(\Phi_p(0,0,\eta,\hat\eta))$ and $D\iota_p(0,0,0,\eta,\hat \eta)\big(\textup{Span}(1, 0, \cT_p^{ss}(\eta,\hat\eta), *,*)\big)$ is close to  $E_X^{ss}(\Phi_p(0,0,\eta,\hat\eta))$, hence the choice of notation. In fact, $\mathcal{T}_p^{ws}$, resp.  $\mathcal{T}_p^{wu}$ can be thought of as the temporal coordinate of the bundle  $E_X^{ws}\oplus E_X^{u}$, resp. $E_X^{s}\oplus E_X^{wu}$, along $W^u_{\mathrm{loc}}(p)$, resp.  $W^s_{\mathrm{loc}}(p)$, in connection with items~\eqref{anomalous_reg1}-\eqref{anomalous_reg2} in Theorem~\ref{claim_improved_thm}.  
\end{remark}
We also denote by $\tau_p\colon \Sigma_p\to \mathbb{R}$ the first return time of the flow $X^t$ on the transversal $\Sigma_p$, specifically, for $x \in \Sigma_p$, $\Pi_p(x)=X^{\tau_p(x)}(x)$. In the following, we identify $\tau_p$ with $\tau_p \circ \Phi_p$. Note that by construction 
\begin{equation}\label{const_time_axes}
\tau_p(\cdot,\cdot,0,0)=\tau_p(0,0,\cdot,\cdot)\equiv T.
\end{equation}

 \subsection{Coarse local coordinates}\label{coarse_charts} While the local sections and charts we have described above are very nice for performing calculations, they do have one drawback: since the linearization procedure depends on higher jets at the fixed point, the sections vary continuously with respect to the flow only in a sufficiently high topology. Meanwhile we have a family of suspension flows which is continuous in $C^1$ topology. Since continuity in parameter $s$
 will play an important role, we will also need to have another family of transversals and charts which would vary continuously in $C^1$ topology with $s$.

 Specifically, we still consider transversals $\Sigma_p$, but impose less rigid conditions on the parametrization $\hat \Phi_p\colon (-1,1)^4\to \Sigma_p$. Namely, we only require the following weaker conditions.
 \begin{enumerate}
 \item $\hat \Phi_p((-1,1)\times(-1,1)\times\{0\}\times\{0\})=W^s_{X,\mathrm{loc}}(p)$; 
  \item\label{item_deux_coarse} $\hat \Phi_p((-1,1)\times\{0\}\times\{0\}\times\{0\})=W^{ss}_{X,\mathrm{loc}}(p)$;
  \item $D\hat\Phi_p(\frac{\partial }{\partial\xi})(p)=E_X^{ws}(p)$;
   \item $\hat \Phi_p(\{0\}\times\{0\}\times(-1,1)\times(-1,1))=W^u_{X,\mathrm{loc}}(p)$; 
   \item $\hat \Phi_p(\{0\}\times\{0\}\times\{0\}\times(-1,1))=W^{uu}_{X\mathrm{loc}}(p)$; 
  \item $D\hat\Phi_p(\frac{\partial }{\partial\eta})(p)=E_X^{wu}(p)$. 
 \end{enumerate}

 We denote by $\hat\Pi_p$ the corresponding Poincar\'e return map, $\hat\Pi_p:=\hat\Phi_p^{-1}\circ \Pi_p\circ\hat\Phi_p$. By the above conditions we have $D\hat\Pi_p(0,0,0,0)=\mathcal L$. We extend the parametrization $\hat \Phi_p$ to an actual local chart $\hat\iota_p$ in exactly the same way as before and we will need the corresponding return time $\hat\tau_p=\tau\circ\hat\Phi_p$. Finally we also define corresponding templates in the exact same way:
\begin{align*}
    D\hat\iota_p(0,0,0,\eta,\hat \eta)\big(\textup{Span}(0, 1, \hat\cT_p^{ws}(\eta,\hat\eta), *,*)\big)&\subset E_X^{s}(\hat \Phi_p(0,0,\eta,\hat\eta)),\\
D\hat\iota_p(0,0,0,\eta,\hat \eta)\big(\textup{Span}(1, 0, \hat\cT_p^{ss}(\eta,\hat\eta), *,*)\big)&\subset E_X^{s}(\hat\Phi_p(0,0,\eta,\hat\eta)).
\end{align*} 

\begin{remark}\label{remark_coarse}
    It will be important that for the $C^1$ family of suspension flows $X^t_s$ (which comes from the $C^2$ family of Anosov diffeomorphisms $F_s$) the coarse charts at $p_s$, $s\in[0,1]$, can be chosen in a continuous manner. Indeed, this is clear because all local invariant manifolds at through $p_s$ vary continuously in $C^1$ topology on the flows and all that is required of the coarse charts $\hat \Phi_{p_s}$ is that they align in certain way with these invariant manifolds.

    In fact the proof simplifies quite a bit for smooth families, since then for the most part the reader can pretend that coarse charts are the same as linearizing charts. Indeed, the main point of coarse charts is that they depend continuously on the parameter and, if the family is smooth, the linearizing charts do depend continuously on the parameter. Still coarse charts are needed for Lemma~\ref{lemma_3.12}.
\end{remark}
 
\subsection{Prescribed families of shadowing periodic orbits}

As above, we consider a center-expanding periodic point $p\in M$, of period $T>0$, with eigenvalues $\hat\mu<\mu<1<\lambda<\hat\lambda$.  

 We fix some homoclinic point $q \in W^u_{\mathrm{loc}}(p)$. We fix a time $T'>0$ with the property that $q'=X^{T'}(q)\in W^s_{\mathrm{loc}}(p)$. Without loss of generality, we can assume that $q,q'\in \Sigma_p$. 
The following shadowing result is standard; see e.g.~\cite[Lemma 4.1]{GLRH} for more details.  
\begin{lemma}\label{lemme_shadow}
	There exist a constant $C_0>0$ and an integer $n_0 \in \mathbb{N}$ such that for $n \geq n_0$, there exists a unique periodic point $p_n\in \Sigma_p$ of period 
	$$
	T_n\simeq nT+T'
	$$
	 such that 
	$$
	d(X^t(p_n),X^t(q))\leq C_0\mu^{\frac n 2},\quad \forall\, t \in \left[\frac{-nT}{2},\frac{nT}{2}+T'\right].
	$$ 
\end{lemma}

\section{The Proof}
This section is devoted to the proof of Proposition~\ref{prop_main}.
Recall that we have reduced the proof of the main theorem to Proposition~\ref{prop_main}.

\subsection{Asymptotic formula for full Jacobian along shadowing periodic orbits (uniform in the family)}

Let $p=X^T(p)$ be a center-expanding periodic point. 
Let $q=\Phi_p(0,0,\eta_\infty,\hat \eta_\infty)\in W^u_{\mathrm{loc}}(p)$ be a homoclinic point, and   fix $T'>0$ such that $q'=X^{T'}(q)=\Phi_p(\hat\xi_\infty,\xi_\infty,0,0)\in W^s_{\mathrm{loc}}(p)$. Let $(p_n=\Phi_p(\hat \xi_n,\xi_n,\eta_n,\hat\eta_n))_n$ be the sequence of shadowing periodic points associated to $q$ given by Lemma~\ref{lemme_shadow}, and let $p_n'=X^{\tau_n}(p_n)=\Phi_p(\hat \xi_n',\xi_n',\eta_n',\hat\eta_n')\in \Sigma_p$ the point in the orbit of $p_n$ which is closest to $q'$. In particular, $\tau_n\approx T'$, $p_n\approx X^{nT}(p_n')$, and by the shadowing lemma, 
\begin{equation}\label{prelim_est}
    (\hat\xi_n,\xi_n,\eta_n,\hat \eta_n)= (0,0,\eta_\infty,\hat \eta_\infty)+O(\mu^{\frac n2}),\quad (\hat\xi_n',\xi_n',\eta_n',\hat \eta_n')= (\hat\xi_\infty,\xi_\infty,0,0)+O(\mu^{\frac n2}).
\end{equation} 
Recall that we denote by $T_n\approx nT+T'$ the period of the point $p_n$. 
\begin{proposition}\label{prop_trois_un}
Given a center-expanding periodic point $p=X^T(p)$, any homoclinic point $q\in W^u_{\mathrm{loc}}(p)$, $q'=X^{T'}(q)\in W^s_{\mathrm{loc}}(p)$, and the corresponding sequence of shadowing periodic points $(p_n)$ associated to $q,q'$, their periods $T_n$ obey the following asymptotic expansion:
    $$
    T_n=nT+T'+\xi_\infty \big(\mathcal{T}_p^{ws}(\eta_\infty,\hat \eta_\infty)-P_p(\eta_\infty,\hat \eta_\infty)\big)  \mu^n+O(\theta^n),
    $$
    with  
$
\theta:=\mu^{\frac{2\log \lambda}{\log \lambda-\log \mu}}\in (\mu^2,\mu)$, and 
    $$
    P_p(\eta_\infty,\hat \eta_\infty)=-\sum_{\ell=1}^{+\infty} \mu^{-\ell} \partial_2 \tau_p(0,0,\lambda^{-\ell} \eta_\infty,\hat \lambda^{-\ell}\hat \eta_\infty),
    $$
    where $\partial_2$ is the partial derivative with respect to the second variable $\xi$.
\end{proposition}
In the following, we abuse notation and denote $\mathcal{T}^{ws}_p(q):=\mathcal{T}_p^{ws}(\eta_\infty,\hat \eta_\infty)$, $P_p(q):=P_p(\eta_\infty,\hat \eta_\infty)$. 
\begin{remark}
Since the conjugacy map $\Phi_p$ between $\Pi_p$ and its linearization is $C^3$, after this change of coordinates the first return time $\tau_p$ is also $C^3$. Moreover, $\partial_2 \tau_p(0,0,0,0)=0$, hence by Taylor expansion, we see that the  terms $P_p^{(\ell)}(\eta_\infty,\hat \eta_\infty):=\mu^{-\ell}\partial_2 \tau_p(0,0,\lambda^{-\ell} \eta_\infty,\hat \lambda^{-\ell}\hat \eta_\infty)$ of the series defining $P_p(\eta_\infty,\hat \eta_\infty)$ decay at least like $(\mu\lambda)^{-\ell}$ (recall $(\mu\lambda)^{-1}<1$ since $p$ is center-expanding), so $(\eta_\infty,\hat \eta_\infty) \mapsto P_p(\eta_\infty,\hat \eta_\infty)$ is indeed a well-defined continuous function. Moreover, for any $\ell \in \mathbb{N}$,  the function $(\eta_\infty,\hat \eta_\infty) \mapsto P_p^{(\ell)}(\eta_\infty,\hat \eta_\infty)$ is $C^1$ (in fact, $C^2$), and both partial derivatives of $P_p^{(\ell)}$  decay faster than $(\mu\lambda)^{-\ell}$, which ensures that the function $P_p$ is, in fact, $C^1$ regular. 
\end{remark}
 
\begin{lemma}\label{prel_est_xin}
    We have the following expansions 
    \begin{align*}
        (\hat \xi_n,\xi_n,\eta_n,\hat \eta_n)-(0,0,\eta_\infty,\hat \eta_\infty)&=(\hat \mu^n \hat \xi_\infty+O(\hat \mu^n\mu^{n}),\mu^n \xi_\infty+O(\mu^{2n}),O(\lambda^{-n}),O(\lambda^{-n})),\\
        (\hat \xi_n',\xi_n',\eta_n',\hat \eta_n')-(\hat\xi_\infty,\xi_\infty,0,0)&=(O(\mu^n),O(\mu^n),\lambda^{-n}\eta_\infty+O(\lambda^{-2n}),\hat \lambda^{-n} \hat \eta_\infty+O(\hat\lambda^{-n}\lambda^{-n})).
    \end{align*}
\end{lemma}
\begin{proof}
    By~\eqref{prelim_est}, and since in the chart $\Phi_p$ the dynamics is given by the linear map $\mathcal{L}$~\eqref{appli_lin}, we have 
    \begin{align}
    (\eta_n',\hat \eta_n')&=(\lambda^{-n} \eta_n,\hat{\lambda}^{-n} \hat\eta_n)=(\lambda^{-n} \eta_\infty+O(\lambda^{-n}\mu^{\frac n2}),\hat{\lambda}^{-n} \hat\eta_\infty+O(\hat\lambda^{-n}\mu^{\frac n2})), \label{est_eta_prime}\\
    (\hat \xi_n,\xi_n)&=(\hat\mu^n \hat\xi_n',\mu^n \xi_n')=(\hat\mu^n \hat\xi_\infty+O(\hat \mu^n\mu^{\frac 12 n}),\mu^n \xi_\infty+O(\mu^{\frac 32 n})). \label{est_xi}
    \end{align}
    Let us denote by $\bar\Pi= X^{\bar \tau}\colon U_q \to U_{q'}$ the Poincar\'e map from a neighborhood $U_q\subset \Sigma_p$ of $q$ to a neighboorhod $U_{q'}\subset \Sigma_p$ of $q'$, with $\bar \tau\approx T'$. Since $\Phi_p^{-1}(W_{\mathrm{loc}}^s(q'))\subset (-1,1)^2 \times \{(0,0)\}$ and $\Phi_p^{-1}(W_{\mathrm{loc}}^u(q'))$ are transverse, estimate~\eqref{est_eta_prime} shows that the unstable distance $d_{W^u}(q',p_n')$ between $q'$ and $p_n'$ is of order $O(\lambda^{-n})$. Applying the $C^3$ diffeomorphism $\bar \Pi^{-1}$, we see that the unstable distance $d_{W^u}(q,p_n)$ between the points $q=\bar \Pi^{-1}(q')$ and $p_n=\bar \Pi^{-1}(p_n')$ is also of order $O(\lambda^{-n})$. Since $\Phi_p^{-1}(W_{\mathrm{loc}}^u(q))\subset \{(0,0)\}\times (-1,1)^2$, we deduce that 
    $$
    (\eta_n,\hat \eta_n)=(\eta_\infty,\hat\eta_\infty)+O(\lambda^{-n}).
    $$
    Using estimate~\eqref{est_xi}, analogous reasoning leads to 
    $$
    (\hat\xi_n',\xi_n')=(\hat\xi_\infty,\xi_\infty)+O(\mu^n). 
    $$
    Applying $\mathcal L^{\pm n}$ to the last two estimates as in~\eqref{est_eta_prime}-\eqref{est_xi} allows us to obtain upgraded estimates 
    \begin{align*}
    (\hat \xi_n,\xi_n)&=(\hat \mu^n \hat \xi_\infty+O(\hat \mu^n\mu^{n}),\mu^n \xi_\infty+O(\mu^{2n})),\\
    (\eta_n',\hat \eta_n')&=(\lambda^{-n} \eta_\infty+O(\lambda^{-2n}),\hat{\lambda}^{-n} \hat\eta_\infty+O(\hat\lambda^{-n}\lambda^{-n})),
    \end{align*}
    which concludes the proof. 
\end{proof}

Again, since the dynamics in the charts is linear~\eqref{appli_lin}, the previous estimates at the entrance/exit points $p_n',p_n$ can be immediately propagated to the whole orbit: 
\begin{corollary}\label{est_prop}
    For any $\ell \in \{0,\cdots,n\}$, let 
    \begin{align*}
    \tilde{p}_n(-\ell)&:=\Phi_p^{-1}(\Pi_p^{-\ell}(p_n))=(\hat\mu^{-\ell}\hat \xi_n,\mu^{-\ell}\xi_n,\lambda^{-\ell}\eta_n,\hat{\lambda}^{-\ell}\hat \eta_n),\\
    \tilde{q}(-\ell)&:=\Phi_p^{-1}(\Pi_p^{-\ell}(q))=(0,0,\lambda^{-\ell}\eta_\infty,\hat{\lambda}^{-\ell}\hat \eta_\infty),\\
    \tilde{q}'(\ell)&:=\Phi_p^{-1}(\Pi_p^{\ell}(q'))=(\hat \mu^\ell\hat\xi_\infty,\mu^\ell\xi_\infty,0,0).
    \end{align*}
    Then, we have 
    \begin{align*}
        \tilde{p}_n(-\ell)-\tilde{q}(-\ell)&=(\hat\mu^{-\ell}\hat \xi_n,\mu^{-\ell}\xi_n,\lambda^{-\ell}\eta_n,\hat{\lambda}^{-\ell}\hat \eta_n)-(0,0,\lambda^{-\ell}\eta_\infty,\hat{\lambda}^{-\ell}\hat \eta_\infty)\\&=(\hat \mu^{n-\ell} \hat \xi_\infty+O(\mu^{n}\hat \mu^{n-\ell}),\mu^{n-\ell} \xi_\infty+O(\mu^{2n-\ell}),O(\lambda^{-(n+\ell)}),O(\lambda^{-{n}}\hat \lambda^{-\ell})),\\
        \tilde{p}_n(-\ell)-\tilde{q}'(n-\ell)&=(\hat\mu^{n-\ell}\hat \xi_n',\mu^{n-\ell}\xi_n',\lambda^{n-\ell}\eta_n',\hat{\lambda}^{n-\ell}\hat \eta_n')-(\hat \mu^{n-\ell}\hat\xi_\infty,\mu^{n-\ell}\xi_\infty,0,0)\\
        &=(O(\mu^{n}\hat \mu^{n-\ell}),O(\mu^{2n-\ell}),\lambda^{-\ell}\eta_\infty+O(\lambda^{-(n+\ell)}),\hat \lambda^{-\ell} \hat \eta_\infty+O(\lambda^{-n}\hat \lambda^{-\ell})).
    \end{align*}
\end{corollary}

\begin{corollary}\label{princ_part_exp}
As $n \to +\infty$, we have the asymptotic expansion 
\begin{equation}\label{asym_exp1}
\sum_{\ell=1}^{n} \tau_p(\tilde{p}_n(-\ell))=n T + \xi_\infty\left[\sum_{\ell=1}^{+\infty}\partial_2 \tau_p(\tilde{q}(-\ell)) \mu^{-\ell}  \right]\mu^n+O(\theta^n),
\end{equation}
where 
$
\theta:=\mu^{\frac{2\log \lambda}{\log \lambda-\log \mu}}\in (\mu^2,\mu)$. 
\end{corollary}
\begin{proof}
For any $n \geq 0$, we define the integer 
\begin{equation}\label{def_elln}
\ell_n:=\left\lfloor\frac{-\log \mu}{\log \lambda-\log \mu}n\right\rfloor,\quad \text{so that }\mu^{n-\ell_n}\approx\lambda^{-\ell_n}. 
\end{equation}
Recall that the roof function $\tau_p$ in normal coordinates is $C^3$. 
By~\eqref{const_time_axes}, 
we have $\partial_3\tau_p(0,0,\cdot,\cdot)=\partial_4\tau_p(0,0,\cdot,\cdot)\equiv 0$; 
similarly, $\partial_1\tau_p(\cdot,\cdot,0,0)=\partial_2\tau_p(\cdot,\cdot,0,0)\equiv 0$. Then, by Corollary~\ref{est_prop}, and by Taylor expansion, for $n \gg 1$ and for any $\ell \in \{1,\cdots,\ell_n\}$, we have
\begin{align*}
\tau_p(\tilde{p}_n(-\ell))&=\tau_p(\tilde{q}(-\ell))+\partial_1 \tau_p(\tilde{q}(-\ell)) (\hat \mu^{n-\ell} \hat \xi_\infty+O(\mu^{n}\hat \mu^{n-\ell}))\\
&+\partial_2 \tau_p(\tilde{q}(-\ell)) (\mu^{n-\ell}  \xi_\infty+O(\mu^{2n-\ell}))+O\left[\sup_{[\tilde{q}(-\ell),\tilde{p}_n(-\ell)]} D^2 \tau_p\left(\tilde{p}_n(-\ell)-\tilde{q}(-\ell)\right)^2\right]\\
&=T+\left[\partial_1 \tau_p(0_{\mathbb{R}^4})+O(\lambda^{-\ell})\right] (\hat \mu^{n-\ell} \hat \xi_\infty+O(\mu^{n}\hat \mu^{n-\ell}))\\
&+\partial_2 \tau_p(\tilde{q}(-\ell)) \mu^{n-\ell}  \xi_\infty+\left[\partial_2\tau_p(0_{\mathbb{R}^4})+O(\lambda^{-\ell})\right]O(\mu^{2n-\ell})\\
&+O\left[\left[D^2 \tau_p(0_{\mathbb{R}^4})+O(\lambda^{-\ell})\right]\left(\tilde{p}_n(-\ell)-\tilde{q}(-\ell)\right)^2\right]\\
&=T+\partial_2 \tau_p(\tilde{q}(-\ell)) \mu^{n-\ell}  \xi_\infty+O(\lambda^{-\ell}\hat \mu^{n-\ell})+O(\lambda^{-\ell}\mu^{2n-\ell})\\
&+O(\mu^{n-\ell}\lambda^{-(n+\ell)})+O(\lambda^{-\ell}\mu^{2(n-\ell)}).
\end{align*}
In the above estimates we have used that:
\begin{itemize}
    \item for $\ell \in \{1,\cdots,\ell_n\}$, the main discrepancy between the periodic point $\tilde{p}_n(-\ell))$ and the homoclinic point $\tilde{q}(-\ell)$ is along the second ($\xi$) coordinate, and it is of order $\mu^{n-\ell}$;
    \item $\partial_1\tau_p(0_{\mathbb{R}^4})=\partial_2\tau_p(0_{\mathbb{R}^4})=0$;
    \item the Hessian $D^2 \tau_p(0_{\mathbb{R}^4})$ has $\partial_1^2 \tau_p(0_{\mathbb{R}^4})=\partial_1\partial_2 \tau_p(0_{\mathbb{R}^4})=\partial_2^2 \tau_p(0_{\mathbb{R}^4})=0$, and similarly, $\partial_3^2 \tau_p(0_{\mathbb{R}^4})=\partial_3\partial_4 \tau_p(0_{\mathbb{R}^4})=\partial_4^2 \tau_p(0_{\mathbb{R}^4})=0$; in other words, its only nonzero terms are those involving a mixed derivative of the form stable direction ($ws$ or $ss$) vs.  unstable direction ($wu$ or $uu$). 
\end{itemize}
Recall that $\mu\lambda>1$, $\hat \mu<\mu<1$, and that $\partial_2 \tau_p(\tilde{q}(-\ell))=O(\lambda^{-\ell})$. We deduce that 
\begin{equation*}
    \sum_{\ell=1}^{\ell_n} \tau_p(\tilde{p}_n(-\ell))=\ell_n T + \xi_\infty\left[\sum_{\ell=1}^{+\infty}\partial_2 \tau_p(\tilde{q}(-\ell)) \mu^{-\ell}  \right]\mu^n+O((\mu\lambda)^{-\ell_n})\mu^n.
\end{equation*}
Note that by~\eqref{def_elln}, we have 
$$
(\mu\lambda)^{-\ell_n}\mu^n\approx \lambda^{-2\ell_n}\approx \mu^{\gamma n},\text{ with }\gamma:=\frac{2\log \lambda}{\log \lambda-\log \mu}\in (1,2).  
$$
Moreover, for any $\ell \in \{\ell_n+1,\cdots,n\}$, we have
\begin{align*}
\tau_p(\tilde{p}_n(-\ell))&=\tau_p(\tilde{q}'(n-\ell))+O\left[\sup_{[\tilde{q}'(n-\ell),\tilde p_n(-\ell)]}\partial_3 \tau_p\right]\lambda^{-\ell}+O\left[\sup_{[\tilde{q}'(n-\ell),\tilde p_n(-\ell)]}\partial_4 \tau_p\right]\hat\lambda^{-\ell}\\
&=T+O(\mu^{n-\ell}\lambda^{-\ell}).
\end{align*}
Gathering together the above estimates, we thus conclude that 
$$
\sum_{\ell=1}^{n} \tau_p(\tilde{p}_n(-\ell))=n T + \xi_\infty\left[\sum_{\ell=1}^{+\infty}\partial_2 \tau_p(\tilde{q}(-\ell)) \mu^{-\ell}  \right]\mu^n+O(\mu^{\gamma n}).\qedhere 
$$ 
\end{proof}
Let us now explain how to conclude the proof of Proposition~\ref{prop_trois_un}. In Corollary~\ref{princ_part_exp}, we have computed the total discrepancy between the period $T$ and the return times to the transverse section $\Sigma_p=\{t=0\}$ of the iterates $\Pi_p^{-\ell}(p_n)$, $\ell\in \{1,\cdots,n\}$, which accounts for the term $\xi_\infty P_p(\eta_\infty,\hat \eta_\infty)\mu^n$ in the asymptotic expansion of the periods $T_n$. It thus remains to deal with the time of the excursion from $p_n$ to $p_n'$, which itself shadows closely the homoclinic excursion from $q$ to $q'$. In particular, this excursion time is of the form $T'+\bar \tau(p_n)$, with $\bar \tau(p_n)=o(1)$. In fact, $\bar \tau(p_n)$ accounts for the term involving the weak stable template $\mathcal{T}_p^{ws}$ in Proposition~\ref{prop_trois_un}:
\begin{lemma}\label{asym_exp_bar-tau}
    As $n \to +\infty$, we have the following asymptotic expansion:
    \begin{equation}
        \bar \tau(p_n)=\xi_\infty\mathcal{T}^{ws}_p(\eta_\infty,\hat \eta_\infty) \mu^n+O(\max(\lambda^{-n},\mu^{2n})).
    \end{equation}
\end{lemma}
\begin{proof}
    Let us sketch the proof of~\eqref{asym_exp_bar-tau}; for more details, we refer the reader to~\cite[Lemma 4.7-Claim 4.8]{GLRH}). The main steps are as follows: 
\begin{itemize}
    \item we claim that $\bar \tau(p_n)$ is equal to the time that the point $p_n$ needs to reach the stable manifold $W_{X,\mathrm{loc}}^s(q)$ (which is~\textit{a priori} not contained in $\Sigma_p$); indeed, the point $X^{\bar \tau(p_n)}(p_n)$ travels together with $q$, hence reaches $\Sigma_p$ exactly after time $T'$;
    \item near the homoclinic point $q$ (i.e., for $|\hat\xi|,|\xi|\ll 1$ in normal coordinates), the stable manifold $W_{X,\mathrm{loc}}^s(q)$ is well approximated by its tangent space $E_X^s(q)$, which in normal coordinates can be expressed as a graph
    \begin{equation}\label{form_graph}
    \big\{(\hat \xi,\xi,\hat\xi\mathcal{T}^{ss}_p(\eta_\infty,\hat \eta_\infty)+\xi\mathcal{T}^{ws}_p(\eta_\infty,\hat \eta_\infty),*,*):\hat \xi,\xi\in \mathbb{R}\big\},
    \end{equation}
    by the definition of the weak-stable and strong stable templates in~\eqref{template-def};
    \item by Lemma~\ref{prel_est_xin}, the leading term of the discrepancy between $p_n$ and $q$ is along the weak-stable direction, i.e., the second coordinate in normal coordinates, and it is of order $\xi_\infty \mu^n+O(\lambda^{-n})$;
    \item since $p_n\in \Sigma_p=\{t=0\}$, by~\eqref{form_graph}, we deduce that 
    $$\bar \tau(p_n)=\xi_\infty\mathcal{T}^{ws}_p(\eta_\infty,\hat \eta_\infty) \mu^n+O(\max(\lambda^{-n},\mu^{2n})),$$
    where the error term $\mu^{2n}$ comes from the approximation of $W_{X,\mathrm{loc}}^s(q)$ by its tangent space $E_X^s(q)$.\qedhere 
\end{itemize}
\end{proof}

Since $T_n=\sum_{\ell=1}^{n} \tau_p(\tilde{p}_n(-\ell))+\bar \tau(p_n)+T'$, combining the asymptotic expansion  of the former sum obtained in Corollary~\ref{princ_part_exp} with the  estimate of $\bar \tau(p_n)$ derived in Lemma~\ref{asym_exp_bar-tau}, this concludes the proof of Proposition~\ref{prop_trois_un}.\qed

\subsection{Reinterpreting the leading term of the asymptotic formula in the coarse charts}\label{sec_reinterpret}

As explained in Subsection~\ref{coarse_charts}, the computations above are performed in linearizing charts which vary continuously with respect to the flow only in a sufficiently high topology (in fact, $C^4$ topology). Here, given a center-expanding point $p\in M$, and homoclinic points $q,q'$ as above, we explain how the coefficient by the leading exponentially small term in the asymptotic formula in Proposition~\ref{prop_trois_un} can be reinterpreted in coarser charts which depend continuously on the flow in $C^1$ topology. We let 
\begin{align*}
    \zeta_p(q)&:=\mathcal{T}_p^{ws}(\eta_\infty,\hat \eta_\infty)-P_p(\eta_\infty,\hat \eta_\infty),\\
    \omega_p(q,q')&:=\xi_\infty \cdot \zeta_p(q), 
\end{align*}
so that $\omega_p(q,q')$ is the coefficient by the leading exponentially small term in the asymptotic formula derived in Proposition~\ref{prop_trois_un}. We let $\hat \Phi_p$, $\hat \Pi_p$, $\hat \tau_p$, and $\hat{\mathcal{T}}_p^{ws}$ be the associated objects in the coarse charts introduced in Subsection~\ref{coarse_charts}. For any $\ell \geq 0$, we let $\hat \mu_{p,q}(-\ell)$ be the stable Jacobian of $D\hat \Pi_p^{-\ell}$ at $\hat \Phi_p^{-1}(q)$, and let 
\begin{align*}
(0,0,\eta_\infty^\circ,\hat \eta_\infty^\circ)&:=\hat \Phi_p^{-1}(q),\\
    (\hat \xi_\infty^\circ,\xi_\infty^\circ,0,0)&:=\hat \Phi_p^{-1}(q'),
    \end{align*}
    \begin{equation}
        \label{eq_series}
    \hat{P}_p(\eta_\infty^\circ,\hat \eta_\infty^\circ):=-\sum_{\ell=1}^{+\infty} \hat \mu_{p,q}(-\ell) \partial_2 \hat\tau_p\big(\hat \Pi_p^{-\ell}(\hat \Phi_p^{-1}(q))\big),
    \end{equation}
    \begin{align*}
    \hat \zeta_p(q)&:=\hat{\mathcal{T}}^{ws}_p(\eta_\infty^\circ,\hat \eta_\infty^\circ)-\hat{P}_p(\eta_\infty^\circ,\hat \eta_\infty^\circ),\\
    \hat{\omega}_p(q,q')&:=\xi_\infty^\circ\cdot \hat \zeta_p(q).
\end{align*}

The following result is precisely the content of~\cite[Lemma 7.16]{GLRH}:
\begin{lemma}\label{lem_coarse}
    There exists a positive constant $\vartheta_p(q)>0$ such that $\zeta_p(q)=\vartheta_p(q)\cdot \hat{\zeta}_p(q)$. 
\end{lemma}
\begin{remark}\label{rmk_coarse}
    In fact, in the exact same way, we can define two functions $\zeta_p,\hat \zeta_p$ along $W_X^u(p)$; moreover, Lemma~\cite[Lemma 7.16]{GLRH} guarantees that there exists a positive function $\vartheta_p\colon W_X^u(p)\to \mathbb{R}_+$ such that $\zeta_p=\vartheta_p\cdot \hat{\zeta}_p$. 
\end{remark}
\begin{corollary}\label{coro_coarse}
    The quantity $\hat{\omega}_p(q,q')$ is $>0$, resp. $=0$, resp. $<0$ if and only if the quantity $\omega_p(q,q')$ is $>0$, resp. $=0$, resp. $<0$. 
\end{corollary}
\begin{proof}
    It follows directly from Lemma~\ref{lem_coarse}; indeed, $\xi_\infty$ and $\xi_\infty^\circ$ have the same sign, since they encode the position of $q'\in W_{X,\mathrm{loc}}^s(p)$ relative to $W_{X,\mathrm{loc}}^{ss}(p)$ in the respective charts $\Phi_p$, $\hat\Phi_p$. 
\end{proof}


\subsection{Non-vanishing of the leading term after a $C^\infty$ perturbation}\label{section_nonvanish}

\begin{proposition}\label{prop_trois_quatre}
    Fix an Anosov diffeomorphism $\bar F$, a center-expanding periodic point $p$ for the suspension flow as above. Fix any 
    $(\eta,\hat\eta)\in(-1,1)^2$ such that $q=\hat\Phi_p(0,0,\eta,\hat\eta)$ is homoclinic. For any $\varepsilon>0$, there exists a $C^\infty$ diffeomorphism $\hat F$, $d_{C^\infty}(\bar F,\hat F)<\varepsilon$, and $r>0$, such that for any $C^\infty$ diffeomorphism $G$ satisfying  $d_{C^2}(G,\hat F)<r$, denoting by $Y^t$ the suspension flow of $G$, and by $p^G$ the  continuation of $p$, the functions $\hat{\mathcal{T}}_{p^G}^{ws}$ and $\hat{P}_{p^G}$ associated to $Y^t$  (in the coarse chart $\hat{\Phi}_{p ^G}$) satisfy
$$
\hat{\mathcal{T}}_{p^G}^{ws}(\eta',\hat \eta')-\hat{P}_{p^G}(\eta',\hat \eta')\neq 0,\quad \forall\, (\eta',\hat \eta')\in B\big((\eta,\hat\eta),r\big). 
$$
\end{proposition} 
The notation for the diffeomorphisms $\bar F$, $\hat F$ and $G$ in Proposition~\ref{prop_trois_quatre} is consistent with the notation in Proposition~\ref{prop_main}. 
\begin{remark}
    In particular, for any diffeomorphism $G$ which is sufficiently $C^2$-close to $\hat F$, for any homoclinic point $q^G=\Phi_p^G(0,0,\eta_\infty,\hat\eta_\infty)=\hat{\Phi}_p^G(0,0,\eta_\infty^\circ,\hat\eta_\infty^\circ)$ with $(\eta_\infty^\circ,\hat\eta_\infty^\circ) \in B\big((\eta,\hat\eta),r\big)$, Proposition~\ref{prop_trois_quatre} ensures that $\hat{\zeta}_p(q)=\hat{\mathcal{T}}_{p^G}^{ws}(\eta_\infty^\circ,\hat \eta_\infty^\circ)-\hat{P}_{p^G}(\eta_\infty^\circ,\hat \eta_\infty^\circ)\neq 0$, hence also, by Lemma~\ref{lem_coarse},  $\zeta_p(q)=\mathcal{T}_{p^G}^{ws}(\eta_\infty,\hat \eta_\infty)-P_{p^G}(\eta_\infty,\hat \eta_\infty)\neq 0$. Consequently, in the asymptotic expansion obtained in  Proposition~\ref{prop_trois_un}, the coefficient $\omega_p(q,q')=\xi_\infty\big(\mathcal{T}_{p^G}^{ws}(\eta_\infty,\hat \eta_\infty)-P_{p^G}(\eta_\infty,\hat \eta_\infty)\big)$ by the leading exponential term  is non-zero unless $\xi_\infty=0$. 
\end{remark}

We split the proof of this proposition into two lemmata, showing respectively the $C^\infty$-density and $C^2$-openness of the above property. 

\begin{lemma}\label{lemma_3.12}
There exists a $C^\infty$-small perturbation $\hat F$ of $\bar F$ such that the functions $\hat{\mathcal{T}}_{p, \hat F}^{ws}$ and $\hat{P}_{p, \hat F}$ associated to the suspension flow of $\hat F$ (in the coarse chart $\hat{\Phi}_{ p}$) satisfy
$$
\hat{\mathcal{T}}_{p, \hat F}^{ws}(\eta,\hat \eta)-\hat{P}_{p,\hat F}(\eta,\hat \eta)\neq 0.
$$  
\end{lemma}

\begin{proof}[Sketch of the proof] We only explain the basic mechanism of non-vanishing without carrying out detailed computations. This is because the perturbative argument is rather standard, in particular it is used to show that strong stable distribution of an Anosov flows is typically not $C^1$ regular, see, e.g.~\cite[Lemma 7.12]{GLRH}.

The perturbation $\hat F$ will be localized at a homoclinic point which corresponds to the preimage of the homoclinic point, i.e., $\hat\Pi^{-1}(0,0,\eta, \hat\eta)$. Further, the perturbation can be arranged so that  the periodic point $p$ and all invariant manifolds through $p$ (and dynamics on them) remain the same for $\hat F$. Then we can use exactly the same coarse chart $\hat\Phi_p$ for both $\bar F$ and $\hat F$ because it is only adapted with respect to these invariant manifolds.

The template $\hat\cT^{ws}_p(\eta,\hat\eta)$ is defined via the chart and the stable subspace $E_X^s(q)=E_X^s(\hat\Phi_p(0,0,\eta,\hat\eta))$. Hence, $\hat\cT^{ws}_p(\eta,\hat\eta)$ is completely determined by the future of $q$ and the chart $\hat\Phi_p$. Since the forward orbit of $q$ is asymptotic to $p$, it never intersects a small neighborhood of $\hat\Phi_p(\hat\Pi^{-1}(0,0,\eta, \hat\eta))$, the stable manifold for such localized perturbation at $q$ remains the same; and since we are still using the same chart, we conclude that the template remains exactly the same for the perturbation $\hat F$:
$$
\hat\cT^{ws}_{p, \bar F}(\eta,\hat\eta)=\hat\cT^{ws}_{p,\hat F}(\eta,\hat\eta).
$$

Now recall that $\hat{P}_{p}$ is given by the series~\eqref{eq_series} which go into the past along the local unstable manifold of $p$. The only term of these series which could change under such localized perturbation is
$$ 
\hat\mu_{p,q}(-1)\partial_2\hat\tau(\hat\Pi^{-1}_p(0,0,\eta,\hat\eta)).
$$
Indeed, the value of this first term can be easily changed by a $C^\infty$-small localized perturbation. For example, we can keep the first jet along the $W^u_{X,\mathrm{loc}}(p)$ the same, hence $\hat\mu_{p,q}(-1)$ will stay the same for $\hat F$. At the same time we can perturb the second jet at $\hat\Pi^{-1}_p(0,0,\eta,\hat\eta)$ changing the derivative of Jacobian at this point with respect to $\xi$, thus changing the value of $\partial_2\hat\tau(\hat\Pi^{-1}_p(0,0,\eta,\hat\eta))
$ (indeed, what we are actually perturbing here is the Jacobian of the diffeomorphism $\bar F$, because the roof function is the logarithmic Jacobian $+\,\mathrm{const}$).
\end{proof}
\begin{remark}
    Note that usage of coarse chart (and corresponding template and series $\hat P_p$) is really helpful for making the perturbation. Otherwise, if we use linearizing chart $\Phi_p$, then perturbing the Jacobian would necessarily result in perturbation of the chart (because it needs to be linearizing for the perturbation), which makes the calculation very delicate.
\end{remark}

\begin{lemma}\label{lemme_stable_nonzero}
Assume that the diffeomorphism $\hat F$ and its associated suspension flow  satisfy
$
\hat{\mathcal{T}}_{\hat p}^{ws}(\hat q)-\hat{P}_{\hat p}(\hat q)\neq 0
$.
Then the same holds for any $C^2$-small perturbation $G$ of $\hat F$, and for all $(\eta',\hat \eta')$ sufficiently close to $(\eta,\hat \eta)$, namely, the functions $\hat{\mathcal{T}}_{p^G}^{ws}$ and $\hat{P}_{p^G}$ for the associated suspension flow of $G$ satisfy 
$$ 
\hat{\mathcal{T}}_{p^G}^{ws}(\eta',\hat \eta')-\hat{P}_{p^G}(\eta',\hat \eta')\neq 0. 
$$
\end{lemma}
\begin{proof}
Recall that the coarser charts $\hat{\Phi}_{p^G}$ depend continuously on the flow in $C^1$ topology, and, accordingly, on the diffeomorphism in $C^2$ topology. 
The template $\hat{\mathcal{T}}_{p^G}^{ws}$ is just a particular coordinate (relative to a continuously varying chart) 
of the stable distribution which depends continuously on the flow in $C^1$ topology on the space of flows (hence, in $C^2$ topology on $G$); and $\hat{P}_{p^G}$ is given by an infinite series involving a first order derivative of the roof function (as in Subsection~\ref{sec_reinterpret}) which is given by the logarithm of the Jacobian of $G$, hence also varies continuously in $C^2$ topology with $G$. Thus, it is clear that the map $(G,(\eta',\hat \eta'))\mapsto \hat\zeta_{p^G}(\eta',\hat \eta')=\hat{\mathcal{T}}_{p^G}^{ws}(\eta',\hat \eta')-\hat{P}_{p^G}
(\eta',\hat \eta')$ 
is continuous in $C^2$ topology on $G$.
\end{proof}

\subsection{Arriving at a contradiction via particular choices of a homoclinic point and parameters}

In this section, we conclude the proof of the main results, i.e., Theorem~\ref{main_theorem} and Theorem~\ref{main_theorem_bis}. 
We start with the proof of the deformation rigidity of typical perturbation of de la Llave examples, and then explain how it can be adapted to show the perturbative rigidity result. 

According to the discussion in Section~\ref{subs_geometric_mechanism}, all that is left to do is to prove Propositions~\ref{prop_main}-\ref{prop_main_tilda}, that is, verifying that given ``generic'' isospectral diffeomorphisms near de la Llave's examples, the conjugacy between them preserves the strong stable/unstable manifolds at some periodic point. This follows immediately from Proposition~\ref{prop_trois_quatre} and the following propositions.

\begin{proposition}\label{prop_main_tech_match}
Consider an Anosov diffeomorphism $G$, a center-expanding periodic point $p^G=Y^T(p^G)$ for its suspension flow $Y^t$, and $(\eta,\hat \eta)\in (-1,1)^2$, $r>0$, as in Proposition~\ref{prop_trois_quatre}, 
such that the functions $\hat{\mathcal{T}}_{p^G}^{ws}$ and $\hat{P}_{p^G}$ associated to $Y^t$ satisfy 
$$
\hat{\mathcal{T}}_{p^G}^{ws}(\eta',\hat \eta')-\hat{P}_{p^G}(\eta',\hat \eta')\neq 0,\quad \forall\, (\eta',\hat \eta')\in B\big((\eta,\hat\eta),r\big). 
$$   
Then for any isospectral deformation $\{G_s\}_{s \in [0,1]}$ based at $G_0=G$, we have $h_s(W_{G, \mathrm{loc}}^{ss}(p^G))= W_{G_s,\mathrm{loc}}^{ss}(p_s^G)$, where $h_s$ is the conjugacy map between $G$ and $G_s$, and $p_s^G$ is the continuation of $p^G$ for $G_s$.  
\end{proposition}
\begin{proof}
In the following, we slightly abuse notation and identify periodic points $p_s^G$ for the diffeomorphisms $G_s$ with the associated periodic points for the flows $Y_s^t$. 
For each $s \in [0,1]$, we abbreviate as $\Sigma_s=\Sigma_{p_s^G}$ a transversal for the suspension flow $Y_s^t$ of $G_s$ containing the local manifolds $W_{Y_s^t,\mathrm{loc}}^s(p_s^G)$, $W_{Y_s^t,\mathrm{loc}}^u(p_s^G)$, we abbreviate as $\Phi_s=\Phi_{p_s^G}$ a chart in which the associated Poincar\'e map is linearized\footnote{Note that the linear map $\mathcal L$ does not depend on the parameter $s$ due to the isospectrality condition.} as in~\eqref{appli_lin}, we abbreviate as $\hat\Phi_s=\hat\Phi_{p_s^G}$ a coarse chart as in Subsection~\ref{sec_reinterpret}, and we denote by $H_s$ the conjugacy between the suspension flows $Y_0^t$ and $Y_s^t$,
$$
H_s \circ Y_0^t=Y_s^t \circ H_s,\quad \forall\, t \in \mathbb{R}. 
$$
Note that all these objects, except  $\Phi_s$, can be chosen to vary continuously with respect to the parameter $s\in [0,1]$, cf. Remark~\ref{remark_coarse}. Moreover, by construction of the parametrizations $\Phi_s$, and similarly for $\hat \Phi_s$, we have $H_s \circ \Phi_0((-1,1)^2 \times \{(0,0)\})\subset  \Phi_s((-1,1)^2 \times \{(0,0)\})$, because the local stable manifold $W_{Y^t,\mathrm{loc}}^s(p^G)\subset \Sigma_0$ is mapped to the local stable manifold $W_{Y_s^t,\mathrm{loc}}^s(p_s^G)\subset \Sigma_s$ by the conjugacy $H_s$. 

Let us consider the set 
$$
\mathcal{G}:=\{s \in [0,1]:H_{s}(W_{Y^t,\mathrm{loc}}^{ss}(p^G))= W_{Y_s^t,\mathrm{loc}}^{ss}(p_s^G)\}. 
$$
We want to show that the set $\mathcal{G}$ is equal to the whole interval $[0,1]$. Given $\hat \xi_0 \in (-1,1)$, we define H\"older continuous functions $[0,1]\ni s\mapsto \hat \xi^\circ(\hat \xi_0,s)$ and $[0,1]\ni s \mapsto \xi^\circ(\hat \xi_0,s)$ implicitly by  
$$
(\hat\xi^\circ(\hat \xi_0,s),\xi^\circ (\hat \xi_0,s),0,0):=\hat\Phi_{s}^{-1} \circ H_{s} \circ \hat\Phi_0(\hat \xi_0,0,0,0),\quad 
\forall\, s\in [0,1].
$$
We argue by contradiction and assume that $\mathcal{G}\neq [0,1]$, i.e.,  
$$
\exists\, s_0\in[0,1]\quad\text{and}\quad\exists\, \hat\xi_0\in (-1,1)\quad\text{such that}\quad
\xi^\circ(\hat \xi_0,s_0)\neq 0.
$$ 
In the following, we fix such $s_0$ and $\hat \xi_0\in (-1,1)$ and abbreviate $\hat\xi^\circ(s):=\hat\xi^\circ(\hat \xi_0,s)$ and $\xi^\circ(s):=\xi^\circ(\hat \xi_0,s)$, for $s \in [0,1]$. 
Without loss of generality, we can assume that $\xi^\circ(s_0)> 0$. 
By the continuity of strong stable manifolds, there exists a small strong stable segment $W_0\subset W_{Y^t,\mathrm{loc}}^{ss}(p^G)$ which contains the point $\hat\Phi_0(\hat \xi_0,0,0,0)$ such that  $H_{s_0}(W_0)$ is disjoint with the strong  stable manifold $W_{Y_{s_0}^t,\mathrm{loc}}^{ss}(p_{s_0}^G)$. See Figure~\ref{fig:delallave}.
By continuity, there exist $\xi_0<0$, and $\varrho>0$ such that for any $(\hat \xi',\xi')\in B((\hat\xi_0,\xi_0),\varrho)$, letting
$(\hat\xi'(s),\xi'(s),0,0):=\hat\Phi_s^{-1}\circ
H_s \circ \hat\Phi_0(\hat \xi',\xi',0,0) 
$, $s\in [0,1]$, we have $\xi'(s)<0$ for $s\approx 0$ and $\xi'(s)>0$ for $s\approx s_0$.  

By the density of homoclinic points, there exists a homoclinic point $q^G=\Phi_0(0,0,\eta_\infty,\hat \eta_\infty)=\hat\Phi_0(0,0,\eta_\infty^\circ,\hat \eta_\infty^\circ) \in W_{Y^t,\mathrm{loc}}^{u}(p^G)$, with $(\eta_\infty^\circ,\hat \eta_\infty^\circ)\in B((\eta,\hat \eta),r)$, such that for some time $T'>0$, the point $(q')^G:=Y^{T'}(q^G)$ satisfies $(q')^G=\Phi_0(\hat \xi_\infty,\xi_\infty,0,0)=\hat\Phi_0(\hat \xi_\infty^\circ,\xi_\infty^\circ,0,0) \in W_{Y^t,\mathrm{loc}}^{s}(p^G)$, with $(\hat \xi_\infty^\circ,\xi_\infty^\circ)\in B((\hat\xi_0,\xi_0),\varrho)$. For any $s \in [0,1]$, let us consider the respective continuations $q_s^G$ and $(q')_s^G$ of the points $q^G$ and $(q')^G$: 
\begin{align*}
q_s^G&=\Phi_s(0,0,\eta_\infty(s),\hat\eta_\infty(s))=\hat\Phi_s(0,0,\eta_\infty^\circ(s),\hat\eta_\infty^\circ(s)) \in W_{Y_s^t,\mathrm{loc}}^{u}(p_s^G),\\
(q')_s^G&=\Phi_s(\hat \xi_\infty(s),\xi_\infty(s),0,0)=\hat\Phi_s(\hat \xi_\infty^\circ(s),\xi_\infty^\circ(s),0,0) \in W_{Y_s^t,\mathrm{loc}}^{s}(p_s^G). 
\end{align*}
In particular, since $(\eta_\infty^\circ,\hat \eta_\infty^\circ)\in B((\eta,\hat \eta),r)$, by Proposition~\ref{prop_trois_quatre}, we have 
$$
\hat\zeta_{p^G}(q^G)=\hat{\mathcal{T}}_{p^G}^{ws}(\eta_\infty^\circ,\hat \eta_\infty^\circ)-\hat{P}_{p^G}(\eta_\infty^\circ,\hat \eta_\infty^\circ)\neq 0,
$$
and for $s \in (0,s_0)$, we have 
$$
\xi_\infty^\circ(s)<0,\, s \approx 0,\quad \xi_\infty^\circ(s)>0,\, s \approx s_0. 
$$  
Since the map $s \mapsto \xi_\infty^\circ(s)$ is continuous (by the continuity of $s \mapsto \hat\Phi_s$), by the intermediate value theorem, we deduce that $\xi_\infty^\circ(s')=0$, for some $s'\in (0,s_0)$. This means that at the parameter value $s'$ the homoclinic point $(q')^G_{s'}$ lies precisely on the strong stable manifold of the point $p^G_{s'}$ as indicated on Figure~\ref{fig:delallave} (recall item~\eqref{item_deux_coarse} from the definition of the coarse chart in Section~\ref{coarse_charts}).

\begin{figure}[!ht]
	\centering
	\includegraphics[width=1.05\textwidth]{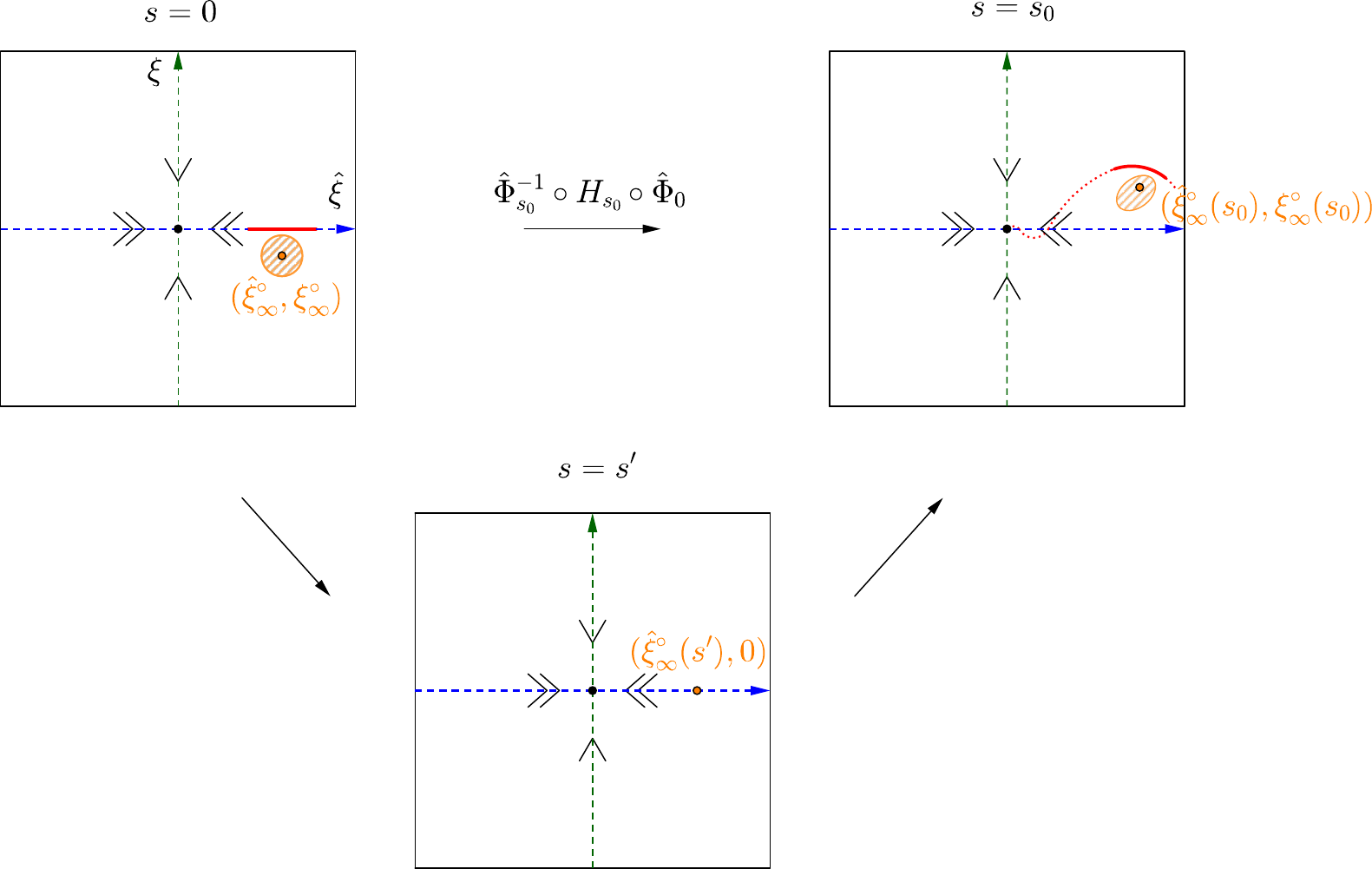}
	\caption{Change of sign of the ``center'' coordinate $\xi_\infty^\circ(s)$ as $s$ changes.}
	\label{fig:delallave}
\end{figure}

Let us now explain how to reach a contradiction. On the one hand, by the (time-preserving) conjugacy between the flows $Y_0^t$ and $Y_{s'}^t$, the periods of the shadowing periodic points $p_n$ and their continuations for $Y_{s'}^t$ are equal to the same numbers $T_n$, which is independent of $s$ (and the same applies to the excursion time $T'$). On the other hand, applying the asymptotic expansion given by  Proposition~\ref{prop_trois_un} to the flow $Y_0^t$,  we obtain 
    $$
    T_n-nT-T'=\omega_{p^G}(q^G,(q')^G) \mu^n+O(\theta^n),
    $$ 
    with $\theta \in (0,\mu)$ and $\omega_{p^G}(q^G,(q')^G):=\xi_\infty \big(\mathcal{T}_{p^G}^{ws}(\eta_\infty,\hat \eta_\infty)-P_{p^G}(\eta_\infty,\hat \eta_\infty)\big)  \neq 0$, while for the flow $Y_{s'}^t$, 
    $$
    T_n-nT-T'=\omega_{p_{s'}^G}(q_{s'}^G,(q')^G_{s'}) \mu^n+O(\theta^n),
    $$
    with $\omega_{p_{s'}^G}(q_{s'}^G,(q')^G_{s'}):=\xi_\infty(s') \big(\mathcal{T}_{p_{s'}^G}^{ws}(\eta_\infty(s'),\hat \eta_\infty(s'))-P_{p_{s'}^G}(\eta_\infty(s'),\hat \eta_\infty(s'))\big)$. Comparing the two expressions above, we thus have
    $$
    \omega_{p^G}(q^G,(q')^G)=\omega_{p_{s'}^G}(q_{s'}^G,(q')^G_{s'})\neq 0. 
    $$
    By Corollary~\ref{coro_coarse}, we deduce that the associated quantity $\hat\omega_{p_{s'}^G}(q_{s'}^G,(q')^G_{s'})$ in coarse charts also satisfies 
    $$
    \hat\omega_{p_{s'}^G}(q_{s'}^G,(q')^G_{s'})=\xi_\infty^\circ(s') \zeta_{p^G_{s'}}(q^G_{s'})\neq 0,
    $$
    yielding a contradiction since $s'$ was chosen such that $\xi_\infty^\circ(s')=0$. We conclude that $\mathcal{G}=[0,1]$, i.e., for any $s \in [0,1]$, $H_{s}(W_{Y^t,\mathrm{loc}}^{ss}(p^G))= W_{Y_s^t,\mathrm{loc}}^{ss}(p_s^G)$. Recall that for each $s\in [0,1]$, the flow $Y_s^t$ is a suspension flow over the Anosov diffeomorphism $G_s$; more precisely, $Y_s^t$ is the flow induced by the vertical flow in $\mathbb{T}^4 \times \mathbb{R}$ on the quotient manifold $M:=(\mathbb{T}^4 \times \mathbb{R})/\hat G_s$, with $\hat G_s\colon (x,\alpha)\mapsto (G_s(x),\alpha-(K+\log J_s(x)))$. In particular,   the local strong stable manifold $W_{Y_s^t,\mathrm{loc}}^{ss}(p^G)$ lifts to a local graph in $\mathbb{T}^4 \times \mathbb{R}$ over the local strong manifold $W_{G_s,\mathrm{loc}}^{ss}(p_s^G)$ for $G_s$. Moreover, the conjugacies $H_s$ between the Anosov flows locally project to conjugacies $h_s$ between the associated Anosov diffeomorphisms, from which we conclude that $h_{s}(W_{G,\mathrm{loc}}^{ss}(p^G))= W_{G_s,\mathrm{loc}}^{ss}(p_s^G)$. 
\end{proof} 

Let us now turn to the following perturbative version:
\begin{proposition}
Consider an Anosov diffeomorphism $G$, a center-expanding periodic point $p^G=Y^T(p^G)$ for its suspension flow $Y^t$. Let $(\eta,\hat \eta)\in (-1,1)^2$ and $r>0$ be as in Proposition~\ref{prop_trois_quatre}, 
such that the functions $\hat{\mathcal{T}}_{p^G}^{ws}$ and $\hat{P}_{p^G}$ associated to $Y^t$ satisfy 
\begin{equation}\label{non_vannish}
\hat{\mathcal{T}}_{p^G}^{ws}(\eta',\hat \eta')-\hat{P}_{p^G}(\eta',\hat \eta')\neq 0,\quad \forall\, (\eta',\hat \eta')\in B\big((\eta,\hat\eta),r\big). 
\end{equation}
Then there exists a $C^2$ open neighborhood $\mathcal{V}$ of $G$ such that for any $G_1\in \mathcal{V}$ which is isospectral to $G$, we have $h(W_{G, \mathrm{loc}}^{ss}(p^G))= W_{ G_1,\mathrm{loc}}^{ss}(p^{G_1})$, where $h$ is the conjugacy map between $G$ and $G_1$, and $p^{G_1}$ is the continuation of $p^G$ for $G_1$.  
\end{proposition}

\begin{proof}
Without loss of generality, we can assume that the quantity in~\eqref{non_vannish} is positive. By 
Lemma~\ref{lemme_stable_nonzero}, we can thus fix a $C^2$ neighborhood $\mathcal{V}$ of $G$ such that for any $G_1\in \mathcal{V}$ conjugate to $G$, if we denote by $Y_1^t$ the associated suspension flow, and by $H$ the conjugacy between the suspension flows $Y^t$ and $Y_1^t$, $
H\circ Y^t=Y_1^t \circ H$, then we have 
\begin{equation}\label{non_vanish_tildeg}
    \hat{\mathcal{T}}_{p^{\tilde G}}^{ws}(H(\eta',\hat \eta'))-\hat{P}_{p^{\tilde G}}(H(\eta',\hat \eta'))> 0,\quad \forall\, (\eta',\hat \eta')\in B\big((\eta,\hat\eta),r\big).
\end{equation}
Fix such a diffeomorphism $G_1\in \mathcal{V}$. We take $Y_1^t$, $H$ as above, and abbreviate $p=p^G$ and $p_1=p^{G_1}$ in the following. We denote by $\Sigma,\Sigma_1$ transversals for the suspension flows $Y^t,Y_1^t$ containing the local stable and unstable manifolds at $p$ and $p_1$, respectively. Let $\Phi,\Phi_1$ be the associated linearizing charts, and let $\hat\Phi,\hat\Phi_1$ be the associated coarse charts. 
Assume by contradiction that 
$$
H(W_{Y^t,\mathrm{loc}}^{ss}(p))\neq  W_{Y_1^t,\mathrm{loc}}^{ss}(p_1).
$$
Then, there exists $\hat \xi_0\in (-1,1)$ such that 
$$
(\hat \xi^\circ,\xi^\circ,0,0):=\hat{\Phi}_1^{-1} \circ H \circ \hat{\Phi}(\hat \xi_0,0,0,0)\text{ has }\xi^\circ\neq 0.
$$
Arguing as in the above proof of deformation rigidity, we can then find a homoclinic point $q^G=\Phi(0,0,\eta_\infty,\hat \eta_\infty)=\hat\Phi(0,0,\eta_\infty^\circ,\hat \eta_\infty^\circ) \in W_{Y^t,\mathrm{loc}}^{u}(p)$, with $(\eta_\infty^\circ,\hat \eta_\infty^\circ)\in B((\eta,\hat \eta),r)$, such that for some time $T'>0$, the point $(q')^G:=Y^{T'}(q^G)\in W_{Y^t,\mathrm{loc}}^{s}(p)$ satisfies 
$$
(q')^G=\Phi(\hat \xi_\infty,\xi_\infty,0,0)=\hat{\Phi}(\hat \xi_\infty^\circ,\xi_\infty^\circ,0,0) \in W_{Y^t,\mathrm{loc}}^{s}(p),\text{ with }\xi_\infty^\circ<0,
$$
while for the points $q_1:=H(q^G)\in W_{Y_1^t,\mathrm{loc}}^{u}(p_1)$ and $q_1':=H((q')^G)=Y_1^{T'}(q_1)\in W_{Y_1^t,\mathrm{loc}}^{s}(p_1)$, we have
$$
q_1'=\Phi_1(\hat \xi_\infty^1,\xi_\infty^1,0,0)=\hat{\Phi}_1(\hat \xi_\infty^{\circ, 1},\xi_\infty^{\circ, 1},0,0),\text{ with }\xi_\infty^{\circ, 1}>0.
$$ 
We then have the following expansions for the (common) periods $T_n$ of the shadowing periodic points, for the flows $Y^t$ and $Y_1^t$:
\begin{align*}
    T_n-nT-T'&= \omega_{p}(q^G,(q')^G)\mu^n+O(\theta^n),\\
    T_n-nT-T'&= \omega_{p_1}(q_1,q_1') \mu^n+O(\theta^n),
\end{align*}
with $\theta\in (0,\mu)$, and 
\begin{align*}
\omega_{p}(q^G,(q')^G)&:=\xi_\infty \big(\mathcal{T}_{p}^{ws}(\eta_\infty,\hat \eta_\infty)-P_{p}(\eta_\infty,\hat \eta_\infty)\big),\\
\omega_{p_1}(q_1,q_1')&:=\xi_\infty^1\big(\mathcal{T}_{p_1}^{ws}(H(\eta_\infty,\hat \eta_\infty))-P_{p_1}(H(\eta_\infty,\hat \eta_\infty))\big).
\end{align*}
In particular, we deduce that $\omega_{p}(q^G,(q')^G)=\omega_{p_1}(q_1,q_1')$. 
Let $\hat\omega_{p}(q^G,(q')^G)$ and $\hat\omega_{p_1}(q_1,q_1')$ be the corresponding quantities in coarse charts. 
Arguing as in the proof of Proposition~\ref{prop_main_tech_match}, we deduce that 
$
\hat\omega_{p}(q^G,(q')^G)$ and $\hat\omega_{p_1}(q_1,q_1')$
have the same sign. But by the definition of the latter quantities, 
by~\eqref{non_vanish_tildeg} (recall that $(\eta_\infty^\circ,\hat \eta_\infty^\circ)\in B((\eta,\hat \eta),r)$), and since $\xi_\infty^\circ<0<\xi_\infty^{\circ,1}$, we have $\hat\omega_{p}(q^G,(q')^G)<0<\hat\omega_{p_1}(q_1,q_1')$, yielding a contradiction. 
\end{proof}

\appendix
\section{On the conjugacy between $5$-dimensional Anosov flows with a fine splitting into $1$-dimensional subbundles}\label{comments_thmc}


In this appendix, we introduce Theorem~\ref{claim_improved_thm}, which extends the results previously established for the restricted class of suspension flows over perturbations of de la Llave’s examples. While we omit the full proof of Theorem~\ref{claim_improved_thm}, we outline the main ideas below.
As we explain, some steps in the proof closely follow the arguments used for suspension flows, while the remaining points can be addressed using reasoning analogous to that in the proof of \cite[Theorem E]{GLRH}.\vspace{0.2cm}

\begin{theoalph}~\label{claim_improved_thm}
    Let $X^t\colon M \to M$ be a $C^\infty$ transitive Anosov flow on a $5$-manifold $M$ which is not a constant roof suspension flow, with a fine dominated splitting into $1$-dimensional subbundles:
    $$
    TM=E_X^{ss}\oplus E_X^{ws}\oplus \mathbb{R} X\oplus E_X^{wu}\oplus E_X^{uu}, 
    $$
    where $E_X^s:=E_X^{ss}\oplus E_X^{ws}$ and $E_X^u:=E_X^{wu}\oplus E_X^{uu}$ are uniformly contracted, respectively expanded, and $X:=\frac{d}{dt}\big|_{t=0} X^t$ is the generating vector field of the flow. For $*\in \{ss,s,u,uu\}$, let $W_X^*$ be the invariant foliation tangent to $E_X^*$. Assume that $X^t$ is ``center-dissipative'', i.e., $\log\det DX^t|_{E^{ws}\oplus E^{wu}}$ is not a coboundary. Then, at least one of the following points 
     is true:
    \begin{enumerate}
        \item\label{anomalous_reg1} 
        for each $x\in M$ the image of the stable subbundle $E_X^s$ in the quotient bundle $TM/E_X^u$, restricted to $W_{X,\mathrm{loc}}^u(x)$ has a  $C^\infty$ section;
        \item\label{anomalous_reg2} 
        for each $x\in M$ the image of the unstable subbundle $E_X^u$ in the quotient bundle $TM/E_X^s$, restricted to $W_{X,\mathrm{loc}}^s(x)$ has a  $C^\infty$ section;
        \item\label{normal_reg} there exists a $C^1$-small neighborhood $\tilde{\mathcal{U}}$ of $X$ such that if a $C^\infty$ Anosov flow $Y^t$ in the neighborhood $\tilde{\mathcal{U}}$  is $C^0$ conjugate to $X^t$, i.e., for some homeomorphism $H$, we have 
   \begin{equation}\label{conj_rel}
   H\circ X^t = Y^t \circ H,\quad \forall\, t \in \mathbb{R},
   \end{equation}
   then:
   \begin{enumerate} 
       \item\label{stepa} $X^t$ and $Y^t$ have the same weak-stable/weak-unstable eigenvalues at periodic points, i.e., for any $p=X^T(p)$, and $*\in \{ws,wu\}$, we have $\det DX^T(p)|_{E_X^{*}}=\det DY^T(H(p))|_{E_Y^{*}}$;
       \item\label{stepb} $H$ preserves strong foliations, i.e., for $*\in \{ss,uu\}$, $H(W_X^*(x))=W_Y^*(H(x))$, $\forall\, x \in M$;
       \item\label{stepc} if, moreover, $X^t$ and $Y^t$ have the same strong-stable/strong-unstable eigenvalues at corresponding periodic points, then $X^t$ and $Y^t$ are $C^{1+\textup{H}}$-conjugate. 
   \end{enumerate}
    \end{enumerate}  
\end{theoalph}


\begin{remark}
    We stress that cases~\eqref{anomalous_reg1}-\eqref{anomalous_reg2} of Theorem~\ref{claim_improved_thm} are ``exceptional''. Indeed, while the bundle $TM/E_X^u$ is smooth along unstable leaves $W_{X,\mathrm{loc}}^u(x)$, $x \in M$, the bundle $E_X^s$ is in general only Hölder continuous along unstable leaves, hence we do not expect the existence of smooth sections as in case~\eqref{anomalous_reg1} (and similarly for case~\eqref{anomalous_reg2}). Note that while case~\eqref{anomalous_reg1} provided some additional partial smoothness of $E^s_X$, we do not claim that these local sections overlap in coherent way to give a global section. It is, however, natural to expect that with further arguments case~\eqref{anomalous_reg1} can be improved to a dichotomy: either image of $E^s_X$ or the image of $E^{ws}_X$ in $TM/E_X^u$ is smooth along the unstable foliation.
\end{remark}
\begin{remark}
Let us also emphasize that the core of this paper focuses on showing items~\eqref{stepb}-\eqref{stepc} within the framework of conjugate suspension Anosov flows over isospectral diffeomorphisms near de la Llave’s examples, under appropriate genericity conditions.
In this context, the eigenvalue matching described in item~\eqref{stepa} is already ensured by the isospectrality of the underlying diffeomorphisms. 
\end{remark}

Below we give some details about the main steps of the proof of Theorem~\ref{claim_improved_thm}. As in the statement of the theorem, we fix a $5$-dimensional transitive Anosov flow $X^t$ which is not a constant roof suspension flow. 

In Subsection~\ref{subs_localcoords}, given a periodic point $p=X^T(p)$ satisfying certain non-resonance conditions, we considered linearizing coordinates for the Poincar\'e map of a smooth transversal $\Sigma_p$ and defined certain ``templates'' $\mathcal{T}_p^{ws},\mathcal{T}_p^{ss}$ along its unstable manifold $W_{X,\mathrm{loc}}^u(p)$. 
The issue is that these objects depend on the non-resonance conditions, hence are not uniform with respect to the periodic point $p$.

To handle this issue, similarly to the strategy in~\cite{GLRH}, we replace linearizing coordinates with less nice normal forms, in order to keep uniformity with respect to the base point, following the construction of adapted charts in the work of Tsujii-Zhang~\cite{tsujiizhang}  (see also~\cite[Proposition 3.2]{GLRH}). More precisely, we can construct a continuous family $\{\imath_x\}_{x \in M}$ of uniformly smooth charts such that the dynamics of $X^t$ in these charts is partially normalized. We associate to these charts a family $\{\tilde{\Sigma}_x\}_{x \in M}$ of uniformly smooth sections transverse to the flow $X^t$, where for each $x \in M$, $\tilde{\Sigma}_x$ contains the local $2$-dimensional stable and unstable manifolds $W_{X,\mathrm{loc}}^s(x),W_{X,\mathrm{loc}}^u(x)$; in particular, the jet of the hitting times of $X^t$ along  $W_{X,\mathrm{loc}}^s(x),W_{X,\mathrm{loc}}^u(x)\subset\tilde \Sigma_x$  is normalized to be a polynomial with uniformly bounded degree. In the same way as in Subsection~\ref{subs_localcoords}, we can then define templates $\tilde{\mathcal{T}}_x^{ws},\tilde{\mathcal{T}}_x^{ss}$ along $W_{X,\mathrm{loc}}^u(x)$, and similarly, templates $\tilde{\mathcal{T}}_x^{wu},\tilde{\mathcal{T}}_x^{uu}$ along $W_{X,\mathrm{loc}}^s(x)$, where these objects are now defined relative to the sections $\tilde{\Sigma}_x$. 

Then, following the proof of~\cite[Proposition 4.2]{GLRH}, given a center-expanding periodic point $p$, homoclinic points $q,q'=X^{T'}(q)$, and for the associated sequence $(p_n)_n$ of periodic points, we can derive asymptotic expansions of their periods $T_n$ similar to those obtained in Proposition~\ref{prop_trois_un}, but in terms of the templates $\tilde{\mathcal{T}}_x^{ws}$ corresponding to the uniform charts: the formula in Proposition~\ref{prop_trois_un} would then become 
$$
    T_n=nT+T'+\tilde{\xi}_\infty \big(\tilde{\mathcal{T}}_p^{ws}(\tilde{\eta}_\infty,\tilde {\hat{\eta}}_\infty)-\tilde{P}_p(\tilde{\eta}_\infty,\tilde {\hat{\eta}}_\infty)\big)  \mu^n+O(\theta^n),
$$
where $\theta \in (0,\mu)$, where $\tilde{P}_p(\tilde{\eta}_\infty,\tilde {\hat{\eta}}_\infty)$ is a polynomial with uniformly bounded degree, and where 
$\tilde{\xi}_\infty,\tilde{\eta}_\infty,\tilde {\hat{\eta}}_\infty$ represent the coordinates of $q,q'$ in $\tilde{\Sigma}_p$. 

Instead of performing a perturbation as in Section~\ref{section_nonvanish} to ensure that the leading term $\tilde{\zeta}_p(q,q'):=\tilde{\mathcal{T}}_p^{ws}(\tilde{\eta}_\infty,\tilde {\hat{\eta}}_\infty)-\tilde{P}_p(\tilde{\eta}_\infty,\tilde {\hat{\eta}}_\infty)$ in the above expansions is non-vanishing, we then argue as in~\cite{GLRH} and show a dichotomy: 
\begin{enumerate}
    \item either the term $\tilde{\zeta}_p(q,q')$ vanishes in a ``robust way'' (i.e., for a dense set of homoclinic points, by varying the points $p,q,q'$), in which case, we can conclude that the templates $\{\tilde{\mathcal{T}}^{ws}_x\}_{x \in M}$ are actually bounded polynomials; 
    this corresponds to  case~\eqref{anomalous_reg1} of Theorem~\ref{claim_improved_thm}. 
    
    Let us outline the latter implication. Denote by $\bar E^s_X$ the image of $E^s_X$ in the quotient bundle $TM/E^u_X$. Recall that given a point $x \in M$ and $y\in W^u_{X,\mathrm{loc}}(x)$, the weak stable template $\tilde{\mathcal T}_x^{ws}(y)$ is defined so that
    $$
    Z_x(y):=(0,1,\tilde{\mathcal T}_x^{ws}(y), *, *)\in E_X^s(y).
    $$
    Accordingly we have 
    the local sections 
    \begin{equation}\label{section_zbarre}
    \bar Z_x\colon W^u_{X,\mathrm{loc}}(x)\to \bar E_X^s,\quad y \mapsto  \bar Z_x(y):=(0,1,\tilde{\mathcal T}_x^{ws}(y))\in\bar E_X^s(y).
    \end{equation}
    Such sections have the following properties:
    \begin{itemize}
        \item since $\{\tilde{\mathcal{T}}^{ws}_x\}_{x\in M}$ are bounded polynomials, for any $x \in M$, the section $\bar Z_x(\cdot)$ is uniformly smooth in $y\in W^u_{X,\mathrm{loc}}(x)$; 
        \item $\textup{span}(\bar Z_x(x))=\bar E_X^{ws}(x)$, where $\bar E_X^{ws}$ is the image of $E_X^{ws}$ in $TM/E^u_X$;
        \item the span of these sections is locally invariant, i.e., for $|t|$ sufficiently small, we have $D\bar X^t(y)(\bar Z_x(y))\in \textup{span}(\bar Z_{X^t(x)}(X^t(y)))$, where $D\bar X^t(y)\colon \bar E_X^s(y)\to \bar E_X^s(X^t(y))$ is the map induced by the differential of the flow $X^t$ between quotient spaces; indeed, in adapted charts, for small $t$, $X^t$ simply becomes the vertical translation by $t$ along the third coordinate (see~\cite[Proposition 3.2]{GLRH});
    \end{itemize}
    \item or for a set of ``positive proportion'' of center-expanding periodic points $p\in M$, we can find $q,q'$ such that $\tilde{\zeta}_p(q,q')\neq 0$, in which case we can deduce several things as in case~\eqref{normal_reg} of Theorem~\ref{claim_improved_thm}:
    \begin{enumerate}
        \item from period expansions, we can recover the weak-stable eigenvalue $\mu=\mu_p\in (0,1)$ for this positive proportion set of periodic points $p$; in particular, by the positive proportion Livshits theorem of~\cite{DilMReber}, if $Y^t$ is another Anosov flow conjugate to $X^t$, with generating vector field sufficiently $C^1$-close to the generating vector field of $X^t$, then we deduce that such eigenvalues have to match for the two flows as in item~\eqref{stepa} of Theorem~\ref{claim_improved_thm};
        \item if $Y^t$ is another Anosov flow conjugate to $X^t$ through a homeomorphism $H$, then, by the non-vanishing of the leading term $\tilde{\zeta}_p(q,q')\neq 0$ for suitable choices of $p,q,q'$, and since such property is robust (similarly to what was shown in Lemma~\ref{lemme_stable_nonzero}),  exactly as we did for suspension flows in the proof of Theorems~\ref{main_theorem}-\ref{main_theorem_bis}, if the generating vector field of $Y^t$ is sufficiently $C^1$-close to the generating vector field of $X^t$, then we can show that the conjugacy $H$ sends the strong stable foliation $W_X^{ss}$ of $X^t$ to the strong foliation $W_Y^{ss}$ of the conjugate flow $Y^t$ as claimed in item~\eqref{stepb} of Theorem~\ref{claim_improved_thm};
        \item finally, from preservation of these foliations, and isospectrality conditions, we can deduce item~\eqref{stepc} of Theorem~\ref{claim_improved_thm} in the same way as we did in the proof of Theorems~\ref{main_theorem}-\ref{main_theorem_bis}. 
    \end{enumerate}
\end{enumerate}
Of course, case~\eqref{anomalous_reg2} of Theorem~\ref{claim_improved_thm} occurs in the symmetric way, when considering center-contracting periodic points instead of center-expanding periodic points. Similarly, information about strong unstable foliations can be deduced from the non-vanishing of the leading exponential term of associated period expansions. 

Let us also stress that the case of ``swapping SRB measures'' which appears in the similar result~\cite[Theorem E]{GLRH} does not occur here, namely in case~\eqref{stepa}, due to the $C^1$-closeness we assume on the vector fields $X,Y$ (see, e.g., the proof of~\cite[Theorem C]{GLRH}
for more details).


\bibliographystyle{alpha}

\bibliography{bib}

\end{document}